\documentclass[journal,twoside,web]{ieeecolor}

\usepackage{etoolbox}
\makeatletter
\@ifundefined{color@begingroup}%
{\let\color@begingroup\relax
	\let\color@endgroup\relax}{}%
\def\fix@ieeecolor@hbox#1{%
	\hbox{\color@begingroup#1\color@endgroup}}
\patchcmd\@makecaption{\hbox}{\fix@ieeecolor@hbox}{}{\FAILED}
\patchcmd\@makecaption{\hbox}{\fix@ieeecolor@hbox}{}{\FAILED}

\usepackage{generic}

\usepackage{lipsum}
\setlipsum{%
	par-before = \begingroup\color{lightgray},
	par-after = \endgroup,
	sentence-before = \begingroup\color{lightgray},
	sentence-after = \endgroup
}

\usepackage{amsthm}

\usepackage{amsmath,amsfonts,amssymb}
\usepackage{algorithm} 
\usepackage{algpseudocode} 
\usepackage{setspace} 
\usepackage{siunitx} 
\usepackage{aligned-overset} 
\usepackage{booktabs} 
\usepackage{multirow} 
\usepackage{xcolor}

\usepackage{tikz,pgfplots}
\usetikzlibrary{arrows,calc,patterns}
\usepgfplotslibrary{groupplots}
\usepackage{tikzscale}
\pgfplotsset{compat=1.18}

\pgfdeclarelayer{background}
\pgfdeclarelayer{foreground}
\pgfsetlayers{background,main,foreground}

\tikzset{
	goalset/.style={semithick,draw=TUMblue,fill=lightblue},
	unsafeset/.style={semithick,draw=red,fill=CORAred},
	set/.style={semithick,draw=black,fill=TUMgray},
	lightgrayset/.style={semithick,draw=black,fill=lightgray}
}

\usepackage{pifont} 

\definecolor{TUMblue}{rgb}{0.00, 0.40, 0.74}
\definecolor{CORAred}{rgb}{0.94510,0.55290,0.56860}%
\definecolor{TUMgray}{rgb}{0.85, 0.85, 0.86}
\definecolor{lightgray}{rgb}{0.92, 0.92, 0.92}
\definecolor{lightblue}{rgb}{0.7529,0.8118,0.9333}
\definecolor{pantone301}{RGB}{0,82,147}
\definecolor{goodRed}{RGB}{227, 27, 35}
\definecolor{goodYellow}{RGB}{255, 195, 37}

\usepackage{mathtools} 

\DeclareMathOperator{\diag}{diag}

\DeclarePairedDelimiter{\norm}{\lVert}{\rVert}

\newcommand{\N}{\mathbb{N}}
\newcommand{\Nint}[2]{\mathbb{N}_{[#1,#2]}}
\newcommand{\R}[1]{\mathbb{R}^{#1}}
\newcommand{\intmat}[1]{\boldsymbol{\mathcal{#1}}}
\newcommand{\bigO}[1]{\mathcal{O}(#1)}
\newcommand{\inputsignals}{\mathbb{U}}
\newcommand{\inputsignalssub}{\widecheckinternal{\mathbb{U}}}
\newcommand{\numInp}{q}
\newcommand{\distsignals}{\mathbb{W}}

\newcommand{\operator}[1]{\normalfont{\mathtt{#1}}}

\newcommand{\boxOp}[1]{\operator{box}(#1)}

\newcommand{\centerOp}[1]{\operator{cen}(#1)}

\newcommand{\convOp}[1]{\operator{conv}(#1)}

\newcommand{\conZonoOp}[1]{\operator{CZ}(#1)}

\makeatletter
\DeclareRobustCommand\widecheckinternal[1]{{\mathpalette\@widecheckinternal{#1}}}
\def\@widecheckinternal#1#2{%
	\setbox\z@\hbox{\m@th$#1#2$}%
	\setbox\tw@\hbox{\m@th$#1%
		\widehat{\vrule\@width\z@\@height\ht\z@
				 \vrule\@height\z@\@width\wd\z@}$}%
	\dp\tw@-2\ht\z@
	\@tempdima\ht\z@ \advance\@tempdima2\ht\tw@ \divide\@tempdima\thr@@
	\setbox\tw@\hbox{\raise1.05\@tempdima\hbox{\scalebox{1}[-1]{\lower\@tempdima\box\tw@}}}%
	{\ooalign{\box\tw@ \cr \box\z@}}}
\makeatother
\newcommand{\widecheck}[1]{\,\widecheckinternal{\kern -2pt #1}}

\DeclareMathOperator{\ooplus}{\widehat{\oplus}}

\renewcommand{\S}[1]{\mathcal{S}_{#1}}

\newcommand{\B}[1]{\mathcal{B}_{#1}}

\newcommand{\I}[1]{\mathcal{I}_{#1}}

\newcommand{\Z}[1]{\mathcal{Z}_{#1}}
\newcommand{\outerZ}[1]{\widehat{\mathcal{Z}}_{#1}}
\newcommand{\innerZ}[1]{\widecheck{\mathcal{Z}}_{#1}}
\DeclarePairedDelimiterX{\zono}[1]{\langle}{\rangle_{Z}}{#1}

\newcommand{\CZ}[1]{\mathcal{C}\thinspace\negthickspace\mathcal{Z}_{#1}}
\DeclarePairedDelimiterX{\conZono}[1]{\langle}{\rangle_{{C}\thinspace\negthickspace{Z}}}{#1}
\newcommand{\cZmat}[1]{K_{#1}}
\newcommand{\cZvec}[1]{l_{#1}}
\newcommand{\cons}[1]{h_{#1}}

\newcommand{\poly}[1]{\mathcal{P}_{#1}}
\DeclarePairedDelimiterX{\polyHRep}[1]{\langle}{\rangle_H}{#1}
\newcommand{\polymat}[1]{H_{#1}}
\newcommand{\polyvec}[1]{d_{#1}}
\newcommand{\polyvectilde}[1]{\tilde{d}_{#1}}

\newcommand{\sF}[1]{\rho\!\left(#1\right)} 
\newcommand{\bigOsF}[1]{\mathtt{SF}(#1)}
\newcommand{\dir}[1]{\ell_{#1}}

\newcommand{\vecones}{\textbf{1}}
\newcommand{\matzeros}{\textbf{0}}

\newcommand{\tFinal}{t_{\text{end}}}
\newcommand{\steps}{\sigma}

\newcommand{\tayl}{\eta}

\newcommand{\grid}{n_{\text{grid}}}

\newcommand{\numTrucks}{\theta}

\newcommand{\trajx}[1]{\xi(#1)}
\newcommand{\F}{\intmat{F}}
\newcommand{\Fu}{\intmat{G}}
\newcommand{\E}{\intmat{E}}

\newcommand{\initset}{\mathcal{X}_{0}}
\newcommand{\targetset}{\mathcal{X}_{\text{end}}}
\newcommand{\targetsetupper}[1]{\mathcal{X}_{\text{end}}^{(#1)}}
\newcommand{\inputset}{\mathcal{U}}
\newcommand{\inputsetupper}[1]{\mathcal{U}^{(#1)}}
\newcommand{\inputsetzero}{\mathcal{U}_0}
\newcommand{\distset}{\mathcal{W}}
\newcommand{\distsetzero}{\mathcal{W}_0}
\newcommand{\constrset}{\bar{\mathcal{X}}}

\newcommand{\Hti}[1]{\mathcal{H}(#1)}

\newcommand{\ZU}[1]{\mathcal{Z}_\mathcal{U}(#1)}
\newcommand{\innerZU}[1]{\widecheck{\mathcal{Z}}_\mathcal{U}(#1)}
\newcommand{\innerZUzero}[1]{\widecheck{\mathcal{Z}}_{\mathcal{U}_0}(#1)}
\newcommand{\outerZU}[1]{\widehat{\mathcal{Z}}_\mathcal{U}(#1)}
\newcommand{\ZW}[1]{\mathcal{Z}_\mathcal{W}(#1)}

\newcommand{\innerZW}[1]{\widecheck{\mathcal{Z}}_\mathcal{W}(#1)}
\newcommand{\outerZW}[1]{\widehat{\mathcal{Z}}_\mathcal{W}(#1)}
\newcommand{\outerZWzero}[1]{\widehat{\mathcal{Z}}_{\mathcal{W}_0}(#1)}

\newcommand{\outerZtraj}[2]{\widehat{\mathcal{Z}}_{#1}(#2)}
\newcommand{\innerZtraj}[2]{\widecheck{\mathcal{Z}}_{#1}(#2)}

\newcommand{\Rti}[1]{\mathcal{R}\!\left(#1\right)}

\newcommand{\BRSE}[1]{\mathcal{R}_{\exists}(#1)}
\newcommand{\BRSEA}[1]{\mathcal{R}_{\exists\forall}(#1)}
\newcommand{\outerBRSE}[1]{\widehat{\mathcal{R}}_{\exists}(#1)}
\newcommand{\innerBRSE}[1]{\widecheck{\mathcal{R}}_{\exists}(#1)}
\newcommand{\innerBRSEA}[1]{\widecheck{\mathcal{R}}_{\exists\forall}(#1)}
\newcommand{\innerBRSEAsuper}[2]{\widecheck{\mathcal{R}}^{(#2)}_{\exists\forall}(#1)}
\newcommand{\outerBRSEA}[1]{\widehat{\mathcal{R}}_{\exists\forall}(#1)}
\newcommand{\BRSEAconstr}[1]{\mathcal{R}_{\exists\forall,\constrset{}}(#1)}

\newcommand{\BRSA}[1]{\mathcal{R}_{\forall}(#1)}
\newcommand{\BRSAE}[1]{\mathcal{R}_{\forall\exists}(#1)}

\newcommand{\innerBRSAE}[1]{\widecheck{\mathcal{R}}_{\forall\exists}(#1)}
\newcommand{\outerBRSAE}[1]{\widehat{\mathcal{R}}_{\forall\exists}(#1)}
\newcommand{\outerBRSAEsuper}[2]{\widehat{\mathcal{R}}^{(#2)}_{\forall\exists}(#1)}

\usepackage[backend=biber,
url=false,
doi=false,
isbn=false,
style=ieee,
sorting=none,
maxnames=3]{biblatex}

\DeclareFieldFormat{titlecase:title}{\MakeSentenceCase*{#1}}
\DeclareFieldFormat{titlecase:booktitle}{#1}
\DeclareFieldFormat{titlecase:journaltitle}{#1}
\usepackage[hidelinks]{hyperref}
\usepackage[capitalise,noabbrev]{cleveref}

\newtheorem{theorem}{Theorem}
\newtheorem{proposition}{Proposition}
\newtheorem{lemma}{Lemma}

\newtheorem{definition}{Definition}

\newtheorem{assumption}{Assumption}

\crefname{table}{Table}{Tables}
\crefname{figure}{Figure}{Figures}
\crefname{algorithm}{Algorithm}{Algorithms}
\crefname{equation}{}{}
\crefname{definition}{Definition}{Definitions}
\crefname{theorem}{Theorem}{Theorems}
\crefname{corollary}{Corollary}{Corollaries}
\crefname{proposition}{Proposition}{Propositions}
\crefname{line}{line}{lines}

\begin{document}

\title{Backward Reachability Analysis of \\ Perturbed Continuous-Time Linear Systems \\ Using Set Propagation}


\author{Mark Wetzlinger and Matthias Althoff
\thanks{
© 2025 IEEE. Author’s version posted under CC BY 4.0. Final published version: https://doi.org/10.1109/TAC.2025.3592719
This work was supported by the European Research Council (ERC) project justITSELF under grant agreement No 817629 and by the German Research Foundation (DFG) under grant numbers GRK 2428 and AL 1185/19-1.}
\thanks{Mark Wetzlinger and Matthias Althoff are with the Department of Computer Engineering, Technical University of Munich, 85748 Garching, Germany (e-mail: \{m.wetzlinger, althoff\}@tum.de).}
}

\maketitle


\begin{abstract}
	Backward reachability analysis computes the set of states that reach a target set under the competing influence of control inputs and disturbances.
	Depending on their interplay, the backward reachable set either represents all states that can be steered into the target set or all states that cannot avoid entering it---the corresponding solutions can be used for controller synthesis and safety verification, respectively.
	A popular technique for backward reachable set computation solves Hamilton-Jacobi-Isaacs equations, which scales exponentially with the state dimension due to gridding the state space.
	Instead, we use set propagation techniques to design backward reachability algorithms for linear time-invariant systems.
	Crucially, the proposed algorithms scale only polynomially with the state dimension.
	Our numerical examples demonstrate the tightness of the obtained backward reachable sets and show an overwhelming improvement of our proposed algorithms over state-of-the-art methods regarding scalability, as systems with well over a hundred state variables can now be analyzed.
\end{abstract}



\section{Introduction}
\label{sec:introduction}

Autonomous systems in safety-critical scenarios require formal verification to rigorously prove safe operation at all times in the presence of uncertainties.
One popular method is backward reachability analysis, which computes the set of states that reach a given target set under a certain interplay between control inputs and disturbances.
This so-called \emph{two-player game} can be set up in two different ways, depending on the meaning of the target set.

If the target set represents an unsafe set, one utilizes the notion of \emph{minimal} reachability \cite[Sec.~4.2]{Mitchell2007HSCC}:
The minimal backward reachable set contains all states that cannot avoid entering the target set regardless of the chosen control input.
Consequently, all states within the backward reachable set are deemed unsafe and thus should be avoided.
In case an exact solution cannot be obtained, we resort to computing outer approximations to maintain safety.
A common example is obstacle avoidance:
The target set represents the obstacle and the minimal backward reachable set contains all states from which one cannot avoid hitting the obstacle.

If the target set represents a goal set, the concept of \emph{maximal} reachability \cite[Sec.~4.1]{Mitchell2007HSCC} is applicable:
The maximal backward reachable set contains all states from which we can steer into the target set despite worst-case disturbances.
Note that any initial state only requires to reach the target set by a single control input trajectory to become part of the backward reachable set.
To ensure that all contained initial states can definitely be steered into the target set, we require an inner approximation if the exact solution cannot be computed.
Maximal backward reachability is closely related to controller synthesis:
The backward reachable set contains all states for which a controller exists such that the target set is reachable.

In this article, we compute minimal and maximal backward reachable sets for continuous-time linear time-invariant (LTI) systems.
As there are many similar definitions of backward reachable sets as well as related concepts, we postpone the literature review to \cref{sec:relatedwork}.
This allows us to use the preliminary information from \cref{sec:preliminaries,sec:problemstatement} for a more concise overview.
Our contributions are as follows:
\begin{itemize}
	\item An inner and outer approximation for the time-point minimal backward reachable set (\cref{ssec:BRSAE_tp}).
	\item An outer approximation of the time-interval minimal backward reachable set (\cref{ssec:outerBRSAE_ti}).
	\item An inner and outer approximation for the time-point maximal backward reachable set (\cref{ssec:BRSEA_tp}).
	\item An inner approximation for the time-interval maximal backward reachable set (\cref{ssec:BRSEA_ti}).
\end{itemize}
Crucially, all proposed algorithms scale only polynomially with respect to the state dimension.
Additionally, we discuss the approximation errors of each computed set.
Our evaluation in \cref{sec:numericalexamples} is followed by closing remarks in \cref{sec:conclusion}.

\section{Preliminaries}
\label{sec:preliminaries}

We introduce some general notation, basics of set-based arithmetic, and fundamentals on forward reachability analysis required for the main body of this article.

\subsection{Notation}
\label{ssec:notation}

The set of real numbers is denoted by $\R{}$, the set of natural numbers without zero is denoted by $\N{}$, and the subset $\{a,a+1,...,b\} \subset \N{}$ for $0 < a < b$, is denoted by $\Nint{a}{b}$.
We denote scalars and vectors by lowercase letters and matrices by uppercase letters.
For a vector $s \in \R{n}$, $\norm{s}_p$ returns its $p$-norm and $s_{(i)}$ represents its $i$th entry;
for a matrix $M \in \R{m \times n}$, $M_{(i,\cdot)}$ refers to the $i$th row and $M_{(\cdot,j)}$ to the $j$th column.
The operation $\diag(s)$ returns a square matrix with the vector~$s$ on its main diagonal.
Horizontal concatenation of two properly-sized matrices $M_1$ and $M_2$ is denoted by $[M_1\;M_2]$ and the identity matrix of dimension~$n$ by $I_n$.
Furthermore, we use $\matzeros{}$ and $\vecones{}$ to represent vectors and matrices of proper dimension containing only zeros or ones.
We denote exact sets by standard calligraphic letters $\S{}$, inner approximations by~$\widecheck{\mathcal{S}} \subseteq \S{}$, and outer approximations by $\widehat{\mathcal{S}} \supseteq \S{}$.
We write the set $\{-s | s \in \S{} \}$ as $-\S{}$ and represent the empty set by $\emptyset$.
An interval is defined by $\I{} = [a,b] = \{ x \in \R{n} \, | \, a \leq x \leq b \}$, where the inequality is evaluated element-wise.
Interval matrices extend intervals by using matrices as lower and upper limits and are denoted in bold calligraphic letters, e.g. $\intmat{I}$.
The operations $\centerOp{\S{}}$ and $\boxOp{\S{}}$ compute the volumetric center and tightest axis-aligned interval outer approximation of the set $\S{}$, respectively.
The Cartesian product of two sets $\S{1}, \S{2}$ is denoted by $\S{1} \times \S{2}$.
Additionally, we introduce the hyperball $\B{\varepsilon} = \{ x \in \R{n} \, | \, \norm{x}_2 \leq \varepsilon \}$.

\subsection{Set-Based Arithmetic}
\label{ssec:setopsandreps}

For convex sets $\S{1}, \S{2} \subset \R{n}$ as well as a matrix $M \in \R{m \times n}$, we formally define the linear map with a matrix and an interval matrix, Minkowski sum, Minkowski difference, intersection, and convex hull:
\begin{align}
	&M \S{1} \coloneqq \{ M s_1 \, | \,  s_1 \in \S{1} \} , \label{eq:def_linMap} \\
	&\intmat{M}\S{1} \coloneqq \{ M s_1 \, | \, M \in \intmat{M}, s_1 \in \S{1} \} \label{eq:def_linMap_intMat} \\
	&\S{1} \oplus \S{2} \coloneqq \{ s_1 + s_2 \, | \, s_1 \in \S{1}, s_2 \in \S{2} \} , \label{eq:def_minkSum} \\
	&\S{1} \ominus \S{2} \coloneqq \{ s \, | \, \{s\} \oplus \S{2} \subseteq \S{1} \} , \label{eq:def_minkDiff} \\
	&\S{1} \cap \S{2} \coloneqq \{ s \, | \,  s \in \S{1} \land s \in \S{2} \} , \label{eq:def_and} \\
	\begin{split} \label{eq:def_conv}
		&\convOp{\S{1},\S{2}} \coloneqq \{ \lambda s_1 + (1-\lambda)s_2 ~ | \\
		&\hspace{90pt} s_1 \in \S{1}, s_2 \in \S{2}, \lambda \in [0,1] \} .
	\end{split}
\end{align}
The support function implicitly describes convex sets:
\begin{definition}[Support function {\cite[Sec.~2]{Girard2008IFAC}}] \label{def:sF}
For a convex, compact set $\S{} \subset \R{n}$ and a vector $\dir{} \in \R{n}$, the support function $\rho: \R{n} \to \R{}$ is
\begin{equation*} 
	\hspace{83pt} \sF{\S{},\dir{}} \coloneqq \max_{s \in \S{}} \dir{}^\top s . \hspace{76pt} \square
\end{equation*}
%
\end{definition}
\noindent For support functions, we require the identities\footnote{Equation (9) in the published version was incorrect and has been removed.} \cite[Eq.~(3)]{Frehse2015CDC}
\begin{align}
	&\sF{M \S{}, \dir{}} = \sF{\S{}, M^\top \dir{}} , \label{eq:linMap_sF} \\
	&\sF{\S{1} \oplus \S{2}, \dir{}} = \sF{\S{1},\dir{}} + \sF{\S{2},\dir{}} . \label{eq:minkSum_sF} 
\end{align}
%
%
Next, we introduce the three set representations required for our backward reachability algorithms.
The runtime complexity of each operation is summarized in \cref{tab:setops}, where we assume a runtime complexity of $\bigO{\max\{p,q\}^{3.5}}$ for the evaluation of a linear program with $p$ variables and $q$ constraints \cite{Boyd2004book}.
We start with polytopes. 
\begin{definition}[Polytope {\cite[Sec.~1.1]{Ziegler2012book}}] \label{def:polytope}
A polytope $\poly{} \subset \R{n}$ in halfspace representation is described using $\cons{} \in \N{}$ linear inequalities defined by the matrix $\polymat{} \in \R{\cons{} \times n}$ and the vector $\polyvec{} \in \R{\cons{}}$:
\begin{equation*} 
	\poly{} \coloneqq \big\{ s \in \R{n} ~ \big| ~ \polymat{} s \leq \polyvec{} \big\} .
\end{equation*}
We use the shorthand $\poly{} = \polyHRep{\polymat{},\polyvec{}}$.
\hfill $\square$
\end{definition}
\noindent A compact set $\S{} \subset \R{n}$ can be enclosed by a polytope through by a finite number $\cons{} \in \N{}$ of support function evaluations
\begin{align}
\begin{split} \label{eq:sFset}
	&\S{} \subseteq \polyHRep{\polymat{},\polyvec{}} , \\
	&\text{where} ~ \forall j \in \Nint{1}{\cons{}}\colon \polyvec{(j)} = \sF{\S{},\polymat{(j,\cdot)}^\top} .
\end{split}
\end{align}
%
Polytopes are closed under all aforementioned set operations \eqref{eq:def_linMap}-\eqref{eq:def_conv} \cite[Tab.~1]{Althoff2021b}.
We will, however, only make use of the linear map with an invertible matrix $M \in \R{n \times n}$ and the Minkowski difference \cite[Thm.~2.2]{Kolmanovsky1998}:
\begin{align}
	&M \poly{} = \polyHRep{\polymat{} M^{-1},\polyvec{}},
	\label{eq:linMap_poly} \\
	&\poly{} \ominus \S{} = \polyHRep{\polymat{}, \tilde{d}},
	\label{eq:minkDiff_poly} \\
	&\text{where}~ \forall j \in \Nint{1}{\cons{}}\colon \tilde{d}_{(j)} = \polyvec{(j)} - \sF{\S{},\polymat{(j,\cdot)}^\top}. \nonumber
\end{align}
%
The enclosing interval $\boxOp{\poly{}}$ is computed via $2n$ support function evaluations (linear programs), one for each column vector in $[I_n \; -\!I_n]$.
Next, we introduce zonotopes.
\begin{definition}[Zonotope {\cite[Def.~1]{Girard2005HSCC}}] \label{def:zonotope}
Given a center $c \in \R{n}$ and $\gamma \in \N{}$ generators stored as columns in the matrix $G \in \R{n \times \gamma}$, a zonotope $\Z{} \subset \R{n}$ is
\begin{equation*} 
	\Z{} \coloneqq \Big\{ c + \sum_{i = 1}^{\gamma} G_{(\cdot,i)} \, \alpha_i ~ \Big| ~ \alpha_i \in [-1,1] \Big\}.
\end{equation*}
We use the shorthand $\Z{} = \zono{c,G}$.
\hfill $\square$
\end{definition}
\noindent For zonotopes, we require the linear map with a matrix $M \in \R{m \times n}$ and Minkowski sum computed as \cite[Eq.~(2.1)]{Althoff2010diss}
\begin{align}
	&M \Z{} = \zono{M c, M G},
	\label{eq:linMap_zono} \\
	&\Z{1} \oplus \Z{2} = \zono{c_1 + c_2, [G_1 ~ G_2]},
	\label{eq:MinkSum_zono}
\end{align}
%
and the support function in a direction $\dir{} \in \R{n}$ \cite[Prop.~1]{LeGuernic2010NAHS}:
\begin{align} 
	\sF{\Z{},\dir{}} &= \dir{}^\top c + \sum_{i=1}^\gamma |\dir{}^\top G_{(\cdot,i)}| .
		\label{eq:sF_zono} 
\end{align}
%
The multiplication of an interval matrix $\intmat{M} = [L,U]$ with a zonotope $\Z{}$ can be enclosed by \cite[Thm.~4]{Althoff2007CDC}
\begin{align}
	&\intmat{M} \Z{} \subseteq
	\big \langle M_c c,	\big[M_c G \; \diag \big( M_r \nu \big) \big] \big \rangle_Z , \label{eq:intmatzono} \\
	&M_c = \tfrac{1}{2}(L+U), M_r = \tfrac{1}{2}(U-L), \nu = |c| + \sum_{i=1}^\gamma | G_{(\cdot,i)}|. \nonumber
\end{align}
%
%
Constrained zonotopes extend zonotopes by introducing equality constraints on the factors.
%
\begin{definition}[Constrained zonotope {\cite[Def.~3]{Scott2016Automatica}}] \label{def:conZonotope}
Given a vector $c \in \R{n}$, a generator matrix $G \in \R{n \times \gamma}$, a constraint matrix $\cZmat{} \in \R{\cons{} \times \gamma}$, and a constraint offset $\cZvec{} \in \R{\cons{}}$, a constrained zonotope $\CZ{}~\subset~\R{n}$ is
\begin{equation*}
	\CZ{} \coloneqq \Big\{ c + \sum_{i = 1}^{\gamma} G_{(\cdot,i)} \, \alpha_i \, \Big| \,
	\sum_{i=1}^{\gamma} \cZmat{(\cdot,i)} \alpha_i = \cZvec{}, \, \alpha_i \in [-1,1] \Big\}.
\end{equation*}
We use the shorthand $\CZ{} = \conZono{c,G,\cZmat{},\cZvec{}}$.
\hfill $\square$
\end{definition}
\noindent For constrained zonotopes, we require the linear map with a matrix $M \in \R{m \times n}$ and Minkowski sum \cite[Prop.~1]{Scott2016Automatica}:
\begin{align*}
	&M \CZ{} = \conZono{Mc, MG, \cZmat{}, \cZvec{}} , \\
	&\CZ{1} \oplus \CZ{2} = \bigg\langle c_1 + c_2, [G_1 ~ G_2], \begin{bmatrix} \cZmat{1} & \matzeros{} \\ \matzeros{} & \cZmat{2} \end{bmatrix}, \begin{bmatrix} \cZvec{1} & \matzeros{} \\ \matzeros{} & \cZvec{2} \end{bmatrix} \bigg\rangle_{CZ} .
\end{align*}
%
The intersection of a constrained zonotope with a polytope $\poly{} = \polyHRep{\polymat{},\polyvec{}}$ can be computed via sequential intersection with each halfspace $\polyHRep{\polymat{(j,\cdot)},\polyvec{(j)}}, j \in \Nint{1}{\cons{}}$ \cite[Thm.~1]{Raghuraman2020Automatica}
\begin{align}
	&\CZ{} \cap \polyHRep{\polymat{(j,\cdot)},\polyvec{(j)}} = \bigg\langle c, [G ~ \matzeros{}],
		\begin{bmatrix} \cZmat{} & \matzeros{} \\ \polymat{(j,\cdot)}G & \tfrac{1}{2} (\polyvec{(j)}-o) \end{bmatrix}, \nonumber \\
	&\hspace{100pt} \begin{bmatrix} \cZvec{} \\ \tfrac{1}{2}(\polyvec{(j)}+o)-\polymat{(j,\cdot)}c \end{bmatrix} \bigg\rangle_{{C}\thinspace\negthickspace{Z}} ,
	\label{eq:intersection_CZ} \\
	&\text{where} \; o = -\sF{\CZ{},-\polymat{(j,\cdot)}^\top} \nonumber
\end{align}
%
evaluates the support function of the constrained zonotope using linear programming. 
%
%
The exact conversion from a polytope to a constrained zonotopes, denoted by $\conZonoOp{\poly{}}$, is computed using \cref{alg:conZonoOp}, which implements \cite[Thm.~1]{Scott2016Automatica}.
%
%
The convex hull can be computed according to \cite[Thm.~5]{Raghuraman2020Automatica} and the multiplication with an interval matrix $\intmat{M} \mathcal{CZ}$ follows from \eqref{eq:intmatzono}.
All introduced set operations scale polynomially in the set dimension and the number of halfspaces/generators, which will enable our backward reachability algorithms to run in polynomial time.

	
	

\begin{algorithm}[!t]
	\caption{Conversion: Polytope to constrained zonotope} \label{alg:conZonoOp}
	\textbf{Require:} Polytope $\poly{} = \polyHRep{\polymat{},\polyvec{}}$
	
	\textbf{Ensure:} Constrained zonotope $\CZ{} = \conZono{c,G,\cZmat{},\cZvec{}}$
	
	\begin{algorithmic}[1]
		\State $\zono{c,G} \gets \boxOp{\poly{}}$
		\State $\forall j \in \Nint{1}{\cons{}}\colon o_{(j)} \gets -\sF{\zono{c,G},-\polymat{(j,\cdot)}^\top}$
		\State $G \gets [G \;\, \matzeros{}]$, $\cZmat{} \gets [\polymat{} G \;\, \tfrac{1}{2}\diag(o - \polyvec{})]$, $\cZvec{} \gets \tfrac{1}{2}(\polyvec{}+o) - \polymat{}c$
		\State $\CZ{} \gets \conZono{c,G,\cZmat{},\cZvec{}}$
	\end{algorithmic}
\end{algorithm}

\begin{table}
\centering \small \setlength{\tabcolsep}{6pt} 
\caption{Runtime complexity of set operations for $n$-dimensional sets\protect\footnotemark.}
\label{tab:setops}
\begin{tabular}{l l c l l}
	\toprule
	\textbf{Operation} & \textbf{Complexity} & & \textbf{Operation} & \textbf{Complexity} \\ \cmidrule{1-2} \cmidrule{4-5}
	$M \poly{}$ & $\bigO{\cons{}n^2}$
		& & $\Z{1} \oplus \Z{2}$ & $\bigO{n}$ \\
	$M\Z{}$ & $\bigO{n^2\gamma}$
	    & & $\CZ{1} \! \oplus \CZ{2}$ & $\bigO{n}$ \\
	$M \CZ{}$ & $\bigO{n^2\gamma}$
		& & $\poly{} \ominus \S{}$ & $\bigO{\cons{} \bigOsF{\S{}}}$ \\
	$\intmat{M} \Z{}$ & $\bigO{n^2\gamma}$
	    & & $\S{} \subseteq \poly{}$ & $\bigO{\cons{} \bigOsF{\S{}}}$ \\
	$\intmat{M}\CZ{}$ & $\bigO{n^2\gamma}$
	    & & $\CZ{} \cap \poly{}$ & $\bigO{\cons{}\gamma^{3.5}}$ \\
	$\sF{\poly{},\dir{}}$ & $\bigO{\cons{}^{3.5}}$
	    & & $\convOp{\CZ{1},\CZ{2}}$ & $\bigO{n}$ \\
	$\sF{\Z{},\dir{}}$ & $\bigO{n\gamma}$
		& & $\boxOp{\poly{}}$ & $\bigO{n\cons{}^{3.5}}$ \\
	$\sF{\CZ{},\dir{}}$ & $\bigO{\gamma^{3.5}}$
		& & $\conZonoOp{\poly{}}$ & $\bigO{n\cons{}^{3.5}}$ \\
	\bottomrule
\end{tabular}
\vspace{-0.5cm}
\end{table}

\footnotetext{The polytope $\poly{}$ has $\cons{} \geq n \in \N{}$ constraints, the constrained zonotope $\CZ{}$ and the zonotope $\Z{}$ have $\gamma \geq n \in \N{}$ generators, $\dir{} \in \R{n}$ is a vector, and $\bigOsF{\S{}}$ denotes the runtime complexity to evaluate $\sF{\S{},\dir{}}$.}


\subsection{Forward Reachable Set Computation}
\label{ssec:FRS}

For an LTI system of the form $\dot{x}(t) = Ax(t) + u(t)$, let the solution trajectory at time~$t \in \R{}$ for an initial state $x_0 \in \initset{} \subset \R{n}$ and an input trajectory $u(\cdot)\colon \R{} \to \inputset{} \subset \R{n}$ be denoted by $\trajx{t;x_0,u(\cdot)}$.
Our backward reachability algorithms leverage established knowledge from forward reachability analysis:
\begin{definition}[Forward reachable set] \label{def:FRS}
The forward reachable set at time $t \geq 0$ is
\begin{equation*}
	\Rti{t} \coloneqq \{ \trajx{t; x_0, u(\cdot)} \, | \, \exists x_0 \in \initset{}, \forall \theta \in [0,t]\colon u(\theta) \in \inputset{} \} . \quad \square
\end{equation*}
\end{definition}

Next, we briefly recall the computation of the homogeneous time-interval solution and the particular solution, which can be computed separately due to the well-known superposition principle of linear systems.

\subsubsection{Homogeneous solution}
\label{sssec:FRS_hom}

Given two homogeneous time-point solutions $\Hti{t_k}, \Hti{t_{k+1}} \subset \R{n}$, we enclose all trajectories over the interval $\tau_k = [t_k,t_{k+1}]$ of length $\Delta t = t_{k+1}-t_k \geq 0$ to enclose the homogeneous time-interval solution \cite[Sec.~3.2]{Althoff2010diss}
\begin{align} 
	\Hti{\tau_k} &\coloneqq \{ e^{At} x(t_k) \, | \, t \in \tau_k, x(t_k) \in \Hti{t_k} \} \label{eq:Hti} \\
	&\subseteq \convOp{\Hti{t_k},\Hti{t_{k+1}}} \oplus \F{} \Hti{t_k} , \label{eq:outerHti}
\end{align}
where the interval matrix $\F{}$ is \cite[Prop.~3.1]{Althoff2010diss}
\begin{equation} \label{eq:F}
	\F{} = \bigoplus_{i=2}^{\tayl{}} \big[\big( i^{\frac{-i}{i-1}} - i^{\frac{-1}{i-1}} \big) \Delta t^i,0\big] \, \frac{A^i}{i!} \oplus \E{} ,
\end{equation}
with $A \in \R{n \times n}$ as in \cref{def:FRS} and the interval matrix
\begin{align}
\begin{split} \label{eq:E}
	\E{} &= [-E(\Delta t,\tayl{}),E(\Delta t,\tayl{})] , \\
	E(\Delta t, \tayl{}) &= e^{|A|\Delta t} - \sum_{i=0}^{\tayl{}} \frac{\big( |A|\Delta t \big)^i}{i!} 
\end{split}
\end{align}
representing the remainder of the exponential matrix \cite[Prop.~2]{Althoff2011HSCC}.
An inner approximation of the homogeneous time-interval solution \eqref{eq:Hti} can be computed by \cite[Prop.~1]{Wetzlinger2023TAC}
\begin{align}
	&\Hti{\tau_k} \supseteq (\convOp{\Hti{t_k},\Hti{t_{k+1}}} \ominus \F{} \Hti{t_k}) \ominus \B{\mu} , \label{eq:innerHti} \\
	&\text{where} \, \mu = \sqrt{\gamma} \, \norm{(e^{A\Delta t} - I_n) G}_2 \label{eq:mu}
\end{align}
uses the generator matrix $G \in \R{n \times \gamma}$ of $\Hti{t_k} = \zono{c,G}$.

\subsubsection{Particular solution}
\label{sssec:FRS_part}

The exact particular solution at time $t = \Delta t$ for time-varying inputs within a set $\S{}$ is defined as
\begin{equation} \label{eq:Z}
	\Z{\S{}}(\Delta t) \coloneqq \bigg\{ \int_{0}^{\Delta t} e^{A(\Delta t-\theta)} s(\theta) \, \mathrm{d}\theta \, \bigg| \, s(\theta) \in \S{} \bigg\} .
\end{equation}
We compute an outer approximation $\outerZ{\S{}}(\Delta t)$ and an inner approximation $\innerZ{\S{}}(\Delta t)$ as \cite[Eq.~(3.7)]{Althoff2010diss}
\begin{align}
	\Z{\S{}}(\Delta t) \subseteq \outerZ{\S{}}(\Delta t) &\coloneqq \bigoplus_{i=0}^\tayl{} \frac{A^i \Delta t^{i+1}}{(i+1)!} \, \S{} \oplus \E{} \Delta t \, \S{}, \label{eq:outerZ} \\
	\Z{\S{}}(\Delta t) \supseteq \innerZ{\S{}}(\Delta t) &\coloneqq A^{-1} (e^{A \Delta t} - I_n) \, \S{} . \label{eq:innerZ}
\end{align}
A suitable value for $\tayl{} \in \N{}$ in \eqref{eq:E} and \eqref{eq:outerZ} can be automatically determined as shown in \cite{Wetzlinger2023TAC}.
For \eqref{eq:innerZ}, we can integrate the term $A^{-1}$ in the power series of the exponential matrix $e^{A\Delta t}$ if the matrix $A$ is not invertible.
The particular solution can be propagated by \cite[Cor.~3.1]{Althoff2010diss} 
\begin{equation} \label{eq:Zprop}
	\Z{\S{}}(t_{k+1}) = \Z{\S{}}(t_k) \oplus e^{At_k} \Z{\S{}}(\Delta t) ,
\end{equation}
which avoids the wrapping effect \cite{Girard2006HSCC1}.
The proposition below states that the inner and outer approximations of the particular solution can be made arbitrarily accurate, which will serve as a key point in later discussions on approximation errors.
\begin{proposition}[Approximation error of particular solution {\cite{Wetzlinger2023TAC}}] \label{prop:convergence}
The Hausdorff distances between the exact particular solution $\Z{\S{}}(t)$ and the outer approximation $\outerZ{\S{}}(t)$ as well as between the exact particular solution $\Z{\S{}}(t)$ and the inner approximation $\innerZ{\S{}}(t)$ converge linearly to $0$, as the time step size $\Delta t$ used in \eqref{eq:Zprop} approaches $0$.
\end{proposition}
\begin{proof}
See \cite[Theorem~1]{Wetzlinger2023TAC}, based on \cite[Lemma~2]{Wetzlinger2023TAC}.
\end{proof}
\noindent For a piecewise constant trajectory, represented as the matrix 
\begin{equation*}
	\setlength{\arraycolsep}{2pt}
	S = \begin{bmatrix} s(t_0) & s(t_1) & \dotsc & s(t_{\steps{}-1}) \end{bmatrix} \in \R{n \times \steps{}}
\end{equation*}
over $\steps{} \in \N{}$ steps,
the particular solution $\Z{s}(\tau_k) \subset \R{n}$ over a time interval $\tau_k$ can be enclosed by \cite[Prop.~3.2]{Althoff2010diss}
%
%
\begin{equation} \label{eq:outerZtraj}
	\outerZtraj{s}{\tau_k} = \bigoplus_{j=0}^{k-1} e^{At_{k-1-j}} \big( A^{-1} (e^{A \Delta t} - I_n) \, s(t_j) \big) \oplus \, \Fu{} \{ s(t_k) \} ,
\end{equation}
where the interval matrix $\Fu{}$ is \cite[Eq.~(3.9)]{Althoff2010diss}
\begin{equation} \label{eq:G}
	\Fu{} = \bigoplus_{i=2}^{\tayl{}+1} \big[\big( i^{\frac{-i}{i-1}} - i^{\frac{-1}{i-1}} \big) \Delta t^i,0\big] \, \frac{A^{i-1}}{i!} \oplus \E{} \Delta t
\end{equation}
with $\E{}$ as in \eqref{eq:E}.
To enclose the particular solution $\Z{\S{}}(\tau_k) \subset \R{n}$ over a time interval $\tau_k$, we first split the set $\S{}$ into two parts \cite[Sec.~3.2.2]{Althoff2010diss}: $\S{} = \S{0} \oplus \{s\}$ with $s = \centerOp{\S{}}$.
Since $\{\matzeros{}\} \in \S{0}$, we have $\outerZ{\S{0}}(\tau_k) \subseteq \outerZ{\S{0}}(t_{k+1})$ and thus
\begin{equation} \label{eq:Znonzero}
	\Z{\S{}}(\tau_k) \subseteq \outerZ{\S{}}(\tau_k) = \outerZ{\S{0}}(t_{k+1}) \oplus \outerZ{s}(\tau_k) ,
\end{equation}
where the set $\outerZ{\S{0}}(t_{k+1})$ is propagated using \eqref{eq:Zprop}, and the set $\outerZ{s}(\tau_k)$ is computed using \eqref{eq:outerZtraj}.

The number of generators required to represent the particular solution \eqref{eq:Zprop} increases with the number of steps.
To mitigate this issue, one can use zonotope order reduction techniques \cite{Yang2016Automatica}---for ease of presentation, however, we omit this operation from our derivations in \cref{sec:backward_minimal,sec:backward_maximal}.

\section{Problem Statement}
\label{sec:problemstatement}

We consider LTI systems of the form
\begin{equation} \label{eq:linsys}
	\dot{x}(t) = Ax(t) + Bu(t) + Ew(t),
\end{equation}
where $x(t) \in \R{n}$ is the state vector, $A \in \R{n \times n}$ is the state matrix, $B \in \R{n \times m}$ is the input matrix, and $E \in \R{n \times r}$ is the disturbance matrix.
The control input $u(t) \in \R{m}$ and the disturbance $w(t) \in \R{r}$ are bounded by the sets $\inputset{} \subset \R{m}$ and $\distset{} \subset \R{r}$, respectively, which we assume to be zonotopes. \linebreak 
We use $\inputsignals{}$ to denote the set of all input trajectories $u(\cdot)$ for which $\forall t \in [0,\tFinal{}]\colon u(t) \in \inputset{}$ holds and analogously $\distsignals{}$ for the set of all disturbances trajectories $w(\cdot)$.
A solution to \eqref{eq:linsys} at time~$t$ starting from the initial state $x_0 \in \R{n}$ using an input trajectory $u(\cdot) \in \inputsignals{}$ and a disturbance trajectory $w(\cdot) \in \distsignals{}$ is written as $\trajx{t;x_0,u(\cdot),w(\cdot)}$.
We denote the particular solutions \eqref{eq:outerZ}-\eqref{eq:Zprop} due to the sets $B \inputset{}$ and $E \distset{}$ at time $t$ by $\ZU{t}$ and $\ZW{t}$, respectively.

In general, backward reachability analysis aims to compute the set of states that reach a target set $\targetset{} \subset \R{n}$ after a certain elapsed time $t$ (time-point backward reachable set) or at any time within the interval $\tau = [t_0,\tFinal{}]$ (time-interval backward reachable set).
We assume the target set $\targetset{} \subset \R{n}$ to be represented as a polytope.

The existing literature, see \cref{sec:relatedwork}, provides varying definitions for minimal and maximal backward reachable sets, depending on the order in which inputs and disturbances are quantified%
\footnote{A review paper on Hamilton-Jacobi reachability \cite{Chen2018ANNUREV} defines backward reachable sets with input-dependent disturbance strategies.
Their \emph{minimal} backward reachable set is defined using 'there exists a disturbance strategy, for which all inputs [...]' \cite[Def.~2]{Chen2018ANNUREV} and their \emph{maximal} backward reachable set using 'for all disturbance strategies, there exists an input [...]' \cite[Def.~1]{Chen2018ANNUREV}.
Another work \cite{Kurzhanskiy2011Automatica} uses \emph{min} and \emph{max} to represent the universal quantifier and existential quantifier, respectively.
Consequently, the \emph{minmax} backward reachable set is defined via 'for all disturbances, there exists an input [...]' \cite[Def.~2.5]{Kurzhanskiy2011Automatica} and the \emph{maxmin} backward reachable sets via 'there exists an input, for which all disturbances [...]' \cite[Def.~2.6]{Kurzhanskiy2011Automatica}.
To avoid any confusion with these existing definitions, we explicitly use the quantifiers in our naming, which also allows for an easy extension to an arbitrary number of quantifiers.}\label{fn:defs}.
We consider the practical case where the input trajectory $u(t)$ is chosen at the start of a time step---and, thus, quantified first---while the disturbance $w(t)$ reacts arbitrarily over the duration of that time step.

Let us first define the AE backward reachable set, where the target set is composed of unsafe states:
\begin{definition}[AE backward reachable set] \label{def:BRSAE}
The time-point AE backward reachable set
\begin{align}
	\begin{split} \label{eq:def_BRSAE_tp}
		\BRSAE{-t} \coloneqq \big\{ x_0 \in \R{n} \, \big| \,
		&\forall u(\cdot) \in \inputsignals{} \; \exists w(\cdot) \in \distsignals{}\colon \\
		&\trajx{t;x_0,u(\cdot),w(\cdot)} \in \targetset{} \big\}
	\end{split}
\end{align}
contains all states, where for all input trajectories $u(\cdot) \in \inputsignals{}$ there is at least one disturbance trajectory $w(\cdot) \in \distsignals{}$ so that the state trajectory will end up in the target set $\targetset{}$ after time~$t$.
The time-interval AE backward reachable set
\begin{align}
	\BRSAE{-\tau} \coloneqq \big\{ x_0 \in \R{n} \, \big| \,
	&\forall u(\cdot) \in \inputsignals{} \; \exists w(\cdot) \in \distsignals{} \; \exists t \in \tau\colon \nonumber \\
	&\trajx{t;x_0,u(\cdot),w(\cdot)} \in \targetset{} \big\} \label{eq:def_BRSAE_ti}
\end{align}
requires the state to pass through $\targetset{}$ anytime in the time interval $\tau$.
\hfill $\square$
\end{definition}

\begin{figure}
	\begin{tikzpicture}[scale=1,
	every node/.style={font=\small}]
	
	
	\node[font=\large] at (0.5,3.5) {\ding{192}};
	
	
	\draw[goalset] (6.5,3.75) -- (7.5,3.75) -- (8,2.75) -- (7,2.25) -- (6,3.25) -- cycle;
	\node[anchor=west,text=TUMblue] at (7.6,3.6) {$\targetset{}$};
	
	\draw[set] (0.5,1.25) -- (1.5,1.75) -- (2,0.75) -- (1,0.25) -- cycle;
	\node[anchor=west] at (0.1,2) {$\BRSAE{-t}$};
	
	\draw[draw=black,fill=TUMgray] (0.8,1.2) circle(2pt) node[xshift=0.3cm,yshift=-0.2cm] {$x_0^{(1)}$};
	\foreach \xend/\yend in {6.05/3.45,6.35/3.25}{ 
		\draw[black] (0.8,1.2) .. controls (2.5,3.25) and (5,2) .. (\xend,\yend);};
	\begin{scope}[xshift=-0.85cm,yshift=0.25cm]
	\draw[black] (6.95,3.25) .. controls (7.2,3.4) and (7.4,3.15) .. (7.2,3) 
		-- (7.2,3) .. controls (6.975,2.875) and (6.75,3.1) .. (6.95,3.25) -- cycle;
	\end{scope}
	\node[anchor=west] at (2,3.5) {$\Rti{t;x_0^{(1)},u^{(1)}(\cdot),\cdot}$};
	\draw (5,3.5) edge[bend left,->,>=stealth',semithick] (6,3.475);
	
	\draw[draw=black,fill=TUMgray] (2.2,1.1) circle(2pt) node[xshift=0.3cm,yshift=-0.2cm] {$x_0^{(2)}$};
	\foreach \xend/\yend in {7.8/2.05,7.675/2.4375}{ 
		\draw (2.2,1.1) .. controls (3.6,2.75) and (6.1,1.55) .. (\xend,\yend);};
	\draw[black,xshift=0.175cm,yshift=-0.21cm] (7.5,2.65) .. controls (7.775,2.8) and (8,2.4) .. (7.6,2.25)
	-- (7.6,2.25) .. controls (7.35,2.175) and (7.275,2.5) .. (7.5,2.65) -- cycle;
	\node[anchor=west] at (3.85,1.5) {$\Rti{t;x_0^{(2)},u^{(2)}(\cdot),\cdot}$};
	\draw (6.8,1.5) edge[bend right,->,>=stealth',semithick] (7.8,2);
	
	
	\begin{scope}[yshift=-4cm]
		
	\node[font=\large] at (0.5,3.5) {\ding{193}};
	
	
	\draw[goalset] (6.5,3.75) -- (7.5,3.75) -- (8,2.75) -- (7,2.25) -- (6,3.25) -- cycle;
	\node[anchor=west,text=TUMblue] at (7.6,3.6) {$\targetset{}$};
	
	\draw[set] (0.5,1.25) -- (1.5,1.75) -- (2,0.75) -- (1,0.25) -- cycle;
	\node[anchor=west] at (0.1,2) {$\BRSEA{-t}$};
	
	\draw[draw=black,fill=TUMgray] (1.1,1.1) circle(2pt) node[xshift=0.3cm,yshift=-0.2cm] {$x_0^{(1)}$};
	\foreach \xend/\yend in {6.95/3.25,7.2/3}{ 
		\draw (1.1,1.1) .. controls (2.5,3.25) and (5,2) .. (\xend,\yend);};
	\draw[black] (6.95,3.25) .. controls (7.2,3.4) and (7.4,3.15) .. (7.2,3) 
		-- (7.2,3) .. controls (6.975,2.875) and (6.75,3.1) .. (6.95,3.25) -- cycle;
	\node[anchor=west] at (2.85,3.5) {$\Rti{t;x_0^{(1)},u^{(1)}(\cdot),\cdot}$};
	\draw (5.75,3.5) edge[bend left,->,>=stealth',semithick] (6.925,3.275);
	
	\draw[draw=black,fill=TUMgray] (2.2,1.1) circle(2pt) node[xshift=0.3cm,yshift=-0.2cm] {$x_0^{(2)}$};
	\foreach \xend/\yend in {7.5/2.65,7.6/2.25}{ 
		\draw[black] (2.2,1.1) .. controls (3.6,2.75) and (6.1,1.55) .. (\xend,\yend);};
	\draw[black] (7.5,2.65) .. controls (7.775,2.8) and (8,2.4) .. (7.6,2.25)
		-- (7.6,2.25) .. controls (7.35,2.175) and (7.275,2.5) .. (7.5,2.65) -- cycle;
	\node[anchor=west] at (3.85,1.5) {$\Rti{t;x_0^{(2)},u^{(2)}(\cdot),\cdot}$};
	\draw (6.8,1.5) edge[bend right,->,>=stealth',semithick] (7.65,2.2);
	
	\end{scope}

\end{tikzpicture}
	\caption{
		Target set $\targetset{}$ with \ding{192} AE backward reachable set $\BRSAE{-t}$ and \ding{193} EA backward reachable set $\BRSEA{-t}$ as well as initial states $x_0$ with corresponding forward reachable sets $\Rti{t}$ for different input trajectories $u(\cdot)$ and disturbance trajectories $w(\cdot)$.
	}
	\label{fig:BRS}
\end{figure}

Case \ding{192} in \cref{fig:BRS} illustrates the time-point set \eqref{eq:def_BRSAE_tp}:
For all states within the AE backward reachable set $\BRSAE{-t}$, such as $x_0^{(1)}$, the target set $\targetset{}$ is unavoidable regardless of the input trajectory $u^{(1)}(\cdot)$.
For any initial state outside $\BRSAE{-t}$ like $x_0^{(2)}$, there is at least one input trajectory $u^{(2)}(\cdot)$, for which there is no disturbance trajectory such that the corresponding forward reachable set intersects $\targetset{}$.

In the following definition of the EA backward reachable set, the target set represents a goal set into which we want to steer the state despite worst-case disturbances.

\begin{definition}[EA backward reachable set] \label{def:BRSEA}
The time-point EA backward reachable set
\begin{align}
\begin{split} \label{eq:def_BRSEA_tp}
	\BRSEA{-t} \coloneqq \big\{ x_0 \in \R{n} \, \big| \,
	&\exists u(\cdot) \in \inputsignals{} \; \forall w(\cdot) \in \distsignals{}\colon \\
	&\trajx{t;x_0,u(\cdot),w(\cdot)} \in \targetset{} \big\}
\end{split}
\end{align}
contains all states, where one input trajectory $u(\cdot)$ can steer the state trajectory into the target set $\targetset{}$ for all potential disturbances $w(\cdot)$.
The time-interval EA backward reachable set
\begin{align}
	\BRSEA{-\tau} \coloneqq \big\{ x_0 \in \R{n} \, \big| \,
	&\exists u(\cdot) \in \inputsignals{} \; \forall w(\cdot) \in \distsignals{} \; \exists t \in \tau\colon \nonumber \\
	&\trajx{t;x_0,u(\cdot),w(\cdot)} \in \targetset{} \big\} \label{eq:def_BRSEA_ti}
\end{align}
requires the state to pass through $\targetset{}$ anytime in the time interval $\tau$.
\hfill $\square$
\end{definition}
Case \ding{193} in \cref{fig:BRS} illustrates the time-point set \eqref{eq:def_BRSEA_tp}:
For all states within the EA backward reachable set $\BRSEA{-t}$, such as $x_0^{(1)}$, there exists an input trajectory $u^{(1)}(\cdot)$ reaching the target set regardless of the disturbance.
In contrast, the forward reachable set of an initial state outside of $\BRSEA{-t}$ like $x_0^{(2)}$ is not contained in the target set for any input trajectory $u^{(2)}(\cdot)$.

Let us briefly highlight an important consequence of the two-player game notion in backward reachability analysis:
\begin{proposition}[Union {\cite[Prop.~2]{Mitchell2007HSCC}}] \label{prop:union}
The union of time-point solutions is a subset of the corresponding time-interval solution, i.e.,
\begin{equation*}
	\bigcup_{t \in \tau} \BRSAE{-t} \subseteq \BRSAE{-\tau} , \quad \bigcup_{t \in \tau} \BRSEA{-t} \subseteq \BRSEA{-\tau} .
\end{equation*}
\end{proposition}
\begin{proof}
This follows from the order of quantifiers \cite[Prop.~2]{Mitchell2007HSCC}.
\end{proof}

For the runtime complexity analysis of our proposed algorithms in \cref{sec:backward_minimal,sec:backward_maximal}, we assume the following:
\begin{assumption}[Parameters] \label{ass:runtimecomplexity}
The number of steps $\steps{}$ and the truncation order $\tayl{}$ in \eqref{eq:outerZ} are fixed, while the number of halfspaces of the target set $\targetset{}$ and the number of generators of the input set $\inputset{}$ and disturbance set $\distset{}$ are linear in the state dimension $n$.
\hfill $\square$
\end{assumption}

In the next section, we review the state of the art in backward reachability analysis.

\section{Related Work}
\label{sec:relatedwork}

A wide range of different yet similar definitions are labeled \emph{backward reachable set}.
The following literature review discusses the various types in order of increasing complexity.
We discuss approaches in discrete and continuous time as well as for linear and nonlinear dynamics, where uniqueness of solution trajectories and sufficient differentiability are assumed.

\subsection{Autonomous Systems}
\label{ssec:autonomous}

The backward reachable set for the dynamics $\dot{x} = f(x)$ is equal to the forward reachable set for the time-inverted dynamics $\dot{x} = -f(x)$ using the target set $\targetset{}$ as the initial set.
If the target set represents an unsafe set, one can use established forward reachability algorithms for computing outer approximations of linear systems \cite{Girard2006HSCC1,LeGuernic2010NAHS} and nonlinear systems \cite{Althoff2008CDC,Chen2015diss}.
If the target set is a goal set, we instead compute an inner approximation, for which there also exist many methods for linear systems \cite{Girard2006HSCC1,Frehse2015CDC} as well as nonlinear systems \cite{Xue2017a,Goubault2017HSCC,Goubault2020CSL,Chen2014FMCAD,Kochdumper2020CDC}.
As this special case is not the focus of our work, we refer the interested reader to the cited literature.

\subsection{Hamilton-Jacobi Reachability}
\label{ssec:HJ}

A well-established framework for computing reachable sets is known as \emph{Hamilton-Jacobi (HJ) reachability}:
It is based on the proof that the reachable set of a continuous-time dynamical system is the zero sublevel set of the Hamilton-Jacobi-Isaacs partial differential equation (PDE) \cite[Thm.~2]{Mitchell2005TAC}.
The value function of the sublevel set is evaluated over a gridded state space, thus the computation scales exponentially with the system dimension \cite{Bansal2017CDC}.
Still, the framework is very versatile, covering the general case of nonlinear dynamics with all variations of competing inputs and disturbances as presented in our subsequent overview of minimal and maximal reachability.

\subsection{Minimal Reachability}
\label{ssec:minimal}

\subsubsection{Unperturbed Case}
\label{sssec:minimal_noW}

Here, \cref{def:BRSAE} simplifies to
\begin{align}
\begin{split} \label{eq:def_BRSA}
	\BRSA{-\tau} \coloneqq \big\{ x_0 \in \R{n} \, \big| \, &\forall u(\cdot) \in \inputsignals{} \; \exists t \in \tau\colon \\
		&\trajx{t;x_0,u(\cdot),\matzeros{}} \in \targetset{} \big\} .
\end{split}
\end{align}
The scalability issue of HJ reachability has first been tackled for time-point solutions by decomposing the dynamics into subsystems and reconstructing the full solution thereafter \cite{Chen2017ICRA}, which was later generalized to time-interval solutions \cite{Chen2018TAC}.
However, these approaches did not provide rigorous results for cases with conflicting controls between subspaces, which was later addressed \cite{Lee2019ICRA}.

\subsubsection{Perturbed Case}
\label{sssec:minimal_withW}

An approach for decoupled dynamics has been presented in \cite{Chen2015CDC}. 
In the context of systems coupled by multi-agent interaction, the decoupled computation has been augmented by a higher-level control using mixed integer programming \cite{Chen2016CDC}.
Moreover, a deep neural network has been trained to output the value function describing the reachable set, which improves the scalability but invalidates all safety guarantees \cite{Bansal2021ICRA}.
Other ideas to improve performance include warm-starting and adaptive grid sampling \cite{Herbert2021ICRA}.


\subsection{Maximal Reachability}
\label{ssec:maximal}

\subsubsection{Unperturbed Case}
\label{sssec:maximal_noW}

Here, \cref{def:BRSEA} simplifies to
\begin{align}
\begin{split} \label{eq:def_BRSE}
	\BRSE{-\tau} \coloneqq \big\{ x_0 \in \R{n} \, \big| \, &\exists u(\cdot) \in \inputsignals{} \; \exists t \in \tau\colon \\
	&\trajx{t;x_0,u(\cdot),\matzeros{}} \in \targetset{} \big\} .
\end{split}
\end{align}
This is equivalent to the forward reachable set for the time-inverted dynamics $\dot{x} = -f(x)$ using the target set as the start set \cite[Lemma~2]{Jones2019CDC}.
As a consequence, all algorithms computing inner approximations for dynamical systems with inputs are applicable, e.g., \cite{Goubault2020CSL} and \cite[Sec.~4.3.3]{Kochdumper2022diss}.
Another approach rescales an initial guess until the forward reachable set is contained in the target set \cite{Schuermann2021TAC2}.
For polynomial systems, sum-of-squares (SOS) optimization is used to compute polynomial lower or upper bounds on the reachable set \cite{Jones2019CDC}.

\subsubsection{Perturbed Case}
\label{sssec:maximal_withW}

Algorithms using set propagation exist for both linear and nonlinear discrete-time systems, with ellipsoids \cite{Kurzhanskiy2011Automatica} or zonotopes \cite{Yang2022CSL,Yang2022TCADICS} as a set representation.
The original HJ reachability was introduced in \cite{Mitchell2005TAC} for the set $\BRSEA{-t}$, with extensions such as decoupling approaches \cite{Chen2018ANNUREV} attempting to alleviate the computational burden.
For dissipative control-affine nonlinear systems, one can reformulate the computation of backward reachable sets as an optimization problem whose variables parametrize a semi-algebraic set representing the reachable set.
By restricting this parametrization to sum-of-square polynomials, one can model the optimization as a semi-definite program, whose number of variables is polynomial in the state dimension, but exponential in the degree of the sum-of-square polynomial \cite{Ahmadi2017CDC}.
The computation of backward reachable sets via SOS programming can be followed by synthesizing a controller to steer the states into the target set \cite{Yin2019ACC}.
This algorithm has been improved by merging both steps into one, including accommodation of control saturation \cite{Yin2021TAC}.
An extension covers a more general class of perturbations represented by integral quadratic constraints \cite{Yin2021CSL}.

An extended definition requires the trajectories to remain within a state constraint set $\constrset{} \subset \R{n}$ at all times:
\begin{align}
\begin{split} \label{eq:def_BRSEA_ti_constr}
	\BRSEAconstr{-\tau} \coloneqq \big\{ &x_0 \in \R{n} \, \big| \,
	\exists u(\cdot) \in \inputsignals{} \; \forall w(\cdot) \in \distsignals{} \\
	&\exists t \in \tau\colon \trajx{t;x_0,u(\cdot),w(\cdot)} \in \targetset{}, \\
	&\forall t' \in [0, \tFinal{}]\colon \trajx{t';x_0,u(\cdot),w(\cdot)} \in \constrset{} \big\} , \nonumber
\end{split}
\end{align}
where $\tFinal{}$ is the upper bound of the time interval $\tau$.
HJ reachability supports this definition \cite{Margellos2011TAC}---including a time-varying state constraint set \cite{Fisac2015HSCC}---as do SOS approaches by solving a single semi-definite program \cite{Xue2020}.

\subsection{Related Concepts}

A related concept is the viability or discriminating kernel:
\begin{definition}[Viability/Discriminating kernel {\cite[Def.~6]{Kaynama2011CDCECC}, \cite[Def.~2]{Kaynama2015TAC}}] \label{def:kernel}
The \emph{discriminating} kernel of a set $\mathcal{K} \subset \R{n}$ is
\begin{align*}
	\mathcal{D}(\tau,\mathcal{K}) \coloneqq \big\{ x_0 \in \mathcal{K} \, \big| \, &\forall w(\cdot) \in \distsignals{} \; \exists u(\cdot) \in \inputsignals{} \; \forall t \in \tau\colon \\
	&\trajx{t;x_0,u(\cdot),w(\cdot)} \in \mathcal{K} \big\} .
\end{align*}
It contains all initial states in $\mathcal{K}$, where for all potential disturbances $w(\cdot)$ there exists an input trajectory $u(\cdot)$ to keep the state in $\mathcal{K}$ over the time interval $\tau$.
Omitting the disturbance $w(\cdot)$ yields the \emph{viability} kernel $\mathcal{V}(\mathcal{K})$.
\hfill $\square$
\end{definition}

Inner approximations of the viability kernel for linear systems can be computed using ellipsoids \cite{Kaynama2011CDCECC,Kaynama2012HSCC} or polytopes \cite{Maidens2013Automatica} as a set representation.
The ellipsoidal methods have later been extended to computing the discriminating kernel in \cite{Kaynama2015TAC}.

Another perspective is the computation of \emph{forward} minimal/maximal reachable sets, where the quantifiers are equal to \cref{def:BRSAE,def:BRSEA}, but the start set is given instead of the target set.
To this end, Kaucher arithmetic has been applied \cite{Goubault2019HSCC} as well as contraction of an outer approximation computed using Taylor models \cite{Goubault2020CSL}.

\smallskip

Our overview of the related literature shows that \cref{def:BRSAE,def:BRSEA} represent general cases of backward reachable sets.
Current approaches using set propagation only deal with discrete-time systems, such as \cite{Kurzhanskiy2011Automatica,Yang2022CSL,Yang2022TCADICS}, while HJ reachability \cite{Chen2018ANNUREV} and SOS approaches \cite{Yin2019ACC,Yin2021TAC,Yin2021CSL} are limited due to exponential complexity.
Subsequently, we present the first propagation-based approach to compute inner and outer approximations of backward reachable sets for continuous-time systems of the form in \eqref{eq:linsys} with polynomial runtime complexity.




\section{Minimal Backward Reachability Analysis}
\label{sec:backward_minimal}

In this section, we compute inner and outer approximations of the time-point AE backward reachable set $\BRSAE{-t}$ given in \eqref{eq:def_BRSAE_tp} in \cref{ssec:BRSAE_tp} as well as an outer approximation of the time-interval AE backward reachable set $\BRSAE{-\tau}$ given in \eqref{eq:def_BRSAE_ti} in \cref{ssec:outerBRSAE_ti}.
We show that the runtime complexity of our algorithms is polynomial in the state dimension $n$, examine the approximation errors, and discuss simplifications for the unperturbed cases $\BRSA{-t}$ and $\BRSA{-\tau}$ defined in \eqref{eq:def_BRSA}.

\subsection{Time-Point Solution}
\label{ssec:BRSAE_tp}

We base our computations of the time-point solution $\BRSAE{-t}$ on the following proposition:
\begin{proposition}[Time-point AE backward reachable set] \label{prop:BRSAE_tp}
The backward reachable set $\BRSAE{-t}$ defined in \eqref{eq:def_BRSAE_tp} can be computed by
\begin{equation} \label{eq:BRSAE_tp}
	\BRSAE{-t} = e^{-At} \big( ( \targetset{} \oplus -\ZW{t} ) \ominus \ZU{t} \big) .
\end{equation}
\end{proposition}
\begin{proof}
See Appendix.
\end{proof}

The formula \eqref{eq:BRSAE_tp} holds independently of the chosen set representations.
Next, we compute approximations in polynomial time assuming a polytopic target set $\targetset{}$ and zonotopic particular solutions $\ZW{t}$ and $\ZU{t}$.

\subsubsection{Outer Approximation}
\label{sssec:outerBRSAE_tp}

The main difficulty in evaluating \eqref{eq:BRSAE_tp} is the Minkowski sum of a polytope in halfspace representation and a zonotope, for which there exists no known polynomial-time algorithm%
\footnote{Other polytope representations, e.g., the Z-representation \cite[Sec.~3.3]{Kochdumper2022diss}, allow for computing the Minkowski sum with a zonotope in polynomial time, but we require the halfspace representation of $(\targetset{} \oplus -\ZW{t})$ for the subsequent computation of the Minkowski difference with $\ZU{t}$ in \eqref{eq:BRSAE_tp}.}.
We overestimate the influence of the disturbance by $\outerZW{t} \supseteq \ZW{t}$ using \eqref{eq:outerZ} and underestimate the influence of the control input by $\innerZU{t} \subseteq \ZU{t}$ using \eqref{eq:innerZ}.
The following proposition provides a scalable yet outer approximative evaluation for the Minkowski sum of a polytope in halfspace representation and a zonotope:
\begin{proposition} (Outer approximation of Minkowski sum) \label{prop:minkSum_polyzono}
Given a polytope $\poly{} = \polyHRep{\polymat{},\polyvec{}} \subset \R{n}$ with $\cons{}$ constraints and a zonotope $\Z{} \subset \R{n}$, their Minkowski sum can be enclosed by
\begin{align}
\begin{split} \label{eq:minkSum_polyzono}
	&\poly{} \oplus \Z{} \subseteq \poly{} \ooplus \Z{} \coloneqq \polyHRep{\polymat{},\polyvec{} + \polyvectilde{}} , \\ 
	&\forall j \in \Nint{1}{\cons{}}\colon \polyvectilde{(j)} = \sF{\Z{},\polymat{(j,\cdot)}^\top} ,
\end{split}
\end{align}
where we introduce the operator $\ooplus$ to distinguish this operation from the exact Minkowski sum.
The runtime complexity is $\bigO{\cons{} n\gamma}$.
\end{proposition}
\begin{proof}
See Appendix.
\end{proof}
The outer approximation in \eqref{eq:minkSum_polyzono} can be further tightened by additional support function evaluations\footnote{In fact, incorporating all infinite directions $\dir{} \in \R{n}$ with $\norm{\dir{}}_2 = 1$ would return the exact result $\poly{} \oplus \Z{}$ since any compact convex set is uniquely determined by the intersection of the support functions in all directions \cite{Girard2008IFAC}.}.
Using \cref{prop:minkSum_polyzono}, we obtain an outer approximation of \eqref{eq:BRSAE_tp} by
\begin{align}
	&\BRSAE{-t} \overset{\eqref{eq:BRSAE_tp}}{=} e^{-At} \big( ( \targetset{} \oplus -\ZW{t} ) \ominus \ZU{t} \big) \nonumber \\
	&\hspace{5pt} \overset{\eqref{eq:outerZ},\,\eqref{eq:innerZ}}{\subseteq} e^{-At} \big( ( \targetset{} \oplus -\outerZW{t} ) \ominus \innerZU{t} \big) \nonumber \\
	&\hspace{1pt} \overset{\text{\cref{prop:minkSum_polyzono}}}{\subseteq} \! e^{-At} \big( ( \targetset{} \ooplus -\outerZW{t} ) \ominus \innerZU{t} \big) \eqqcolon \outerBRSAE{-t} , \label{eq:outerBRSAE_tp}
\end{align}
resulting in a polytope representing $\outerBRSAE{-t}$.

\subsubsection{Inner Approximation}
\label{sssec:innerBRSAE_tp}

We now underestimate the influence of the disturbance by $\innerZW{t} \subseteq \ZW{t}$ and overestimate the influence of the control input by $\outerZU{t} \supseteq \ZU{t}$.
Using the following re-ordering relation for the compact, convex, nonempty sets $\S{1},\S{2},\S{3} \subset \R{n}$ \cite[Lemma~1(i)]{Yang2022CSL}
\begin{equation} \label{eq:reordering}
	( \S{1} \oplus \S{2} ) \ominus \S{3} \supseteq ( \S{1} \ominus \S{3} ) \oplus \S{2} ,
\end{equation}
we can inner approximate \eqref{eq:BRSAE_tp} by
\begin{align}
	&\BRSAE{-t} \overset{\eqref{eq:BRSAE_tp}}{=} e^{-At} \big( ( \targetset{} \oplus -\ZW{t} ) \ominus \ZU{t} \big) \nonumber \\
	&\overset{\eqref{eq:outerZ},\,\eqref{eq:innerZ}}{\supseteq} e^{-At} \big(  ( \targetset{} \oplus -\innerZW{t} ) \ominus \outerZU{t} \big) \nonumber \\
	&\hspace{8pt} \overset{\eqref{eq:reordering}}{\supseteq} e^{-At} \big( \conZonoOp{\targetset{} \ominus \outerZU{t}} \oplus -\innerZW{t} \big) \eqqcolon \innerBRSAE{-t}, \label{eq:innerBRSAE_tp}
\end{align}
where we evaluate $\targetset{} \ominus \outerZU{t}$ by \eqref{eq:minkDiff_poly} and convert the resulting polytope to a constrained zonotope using \cref{alg:conZonoOp} to efficiently evaluate the Minkowski sum with $-\innerZW{t}$.

\subsubsection{Runtime Complexity}
\label{sssec:BRSAE_tp_bigO}

Under \cref{ass:runtimecomplexity} and following \cref{tab:setops}, the outer approximative Minkowski sum from \cref{prop:minkSum_polyzono}, the Minkowski difference, and the linear map in the computation of the outer approximation $\outerBRSAE{-t}$ are all $\bigO{n^3}$, while the computation of the inner approximation $\innerBRSAE{-t}$ is dominated by the conversion to a constrained zonotope, which is $\bigO{n^{4.5}}$.

\subsubsection{Approximation Error}
\label{sssec:BRSAE_tp_eps}

Both approximations have a non-zero approximation error even in the limit $\Delta t \to 0$ due to using \cref{prop:minkSum_polyzono} and the re-ordering in \eqref{eq:reordering}, respectively. 
The approximation error of the more important outer approximation $\outerBRSAE{-t}$ can be made arbitrarily small in all directions selected for the evaluation of \cref{prop:minkSum_polyzono}, as the Hausdorff distance between the computed particular solutions and their exact counterpart goes to $0$ as $\Delta t \to 0$ by \cref{prop:convergence}.


\subsubsection{Unperturbed Case}
\label{sssec:BRSAE_tp_noW}

In the case $\distset{} = \{\matzeros{}\}$, we compute the backward reachable set defined in \eqref{eq:def_BRSA} with $\tau = t$, for which \eqref{eq:outerBRSAE_tp} and \eqref{eq:innerBRSAE_tp} simplify accordingly, resulting in the same respective runtime complexities.
The approximation error depends on the error of the particular solution, which can be made arbitrarily small according to \cref{prop:convergence}.

\subsection{Time-Interval Solution}
\label{ssec:outerBRSAE_ti}

For the time-interval solution $\BRSAE{-\tau}$, we compute an outer approximation enclosing all states that cannot avoid entering the target set $\targetset{}$.
We reformulate the definition in \eqref{eq:def_BRSAE_ti} to
\begin{equation} \label{eq:BRSAE_ti}
	\BRSAE{-\tau} = \bigcap_{u^*(\cdot) \in \inputsignals{}} \BRSE{-\tau;u^*(\cdot)} ,
\end{equation}
where
\begin{align}
	\begin{split} \label{eq:def_BRSE_ustar}
		\BRSE{-\tau;u^*(\cdot)} \coloneqq \big\{ x_0 \in \R{n} \, \big| \,
		&\exists w(\cdot) \in \distsignals{} \; \exists t \in \tau\colon \\
		&\trajx{t;x_0,u^*(\cdot),w(\cdot)} \in \targetset{} \big\}
	\end{split}
\end{align}
is the forward reachable set for the time-inverted dynamics using a single input trajectory $u^*(\cdot) \in \inputsignals{}$, which is equivalent to \eqref{eq:def_BRSE} by replacing $u$ by $w$.
Consequently, the set $\BRSAE{-\tau}$ is the intersection of the sets $\BRSE{-\tau;u^*(\cdot)}$ for all potential input trajectories $u^*(\cdot) \in \inputsignals{}$.

\subsubsection{Outer Approximation}
\label{sssec:outerBRSAE_ti}

Let us first introduce our high-level idea for computing an outer approximation of \eqref{eq:BRSAE_ti}:
Note that the intersection of \emph{any} number of $\BRSE{-\tau;u^*(\cdot)}$ in \eqref{eq:BRSAE_ti} always leads to a sound outer approximation.
Obviously, we want a tight outer approximation, that is, a small intersection stemming from a well selected, finite number of input trajectories $u^*(\cdot) \in \inputsignals{}$.
For this selection, we use a heuristic approach via support function reachability, which simplifies the intersection of the individual reachable sets in \eqref{eq:BRSAE_ti}.

The main two steps of our computation are illustrated in \cref{fig:BRSAE_ti}.
Step 1: We enclose the reachable set $\BRSE{-\tau;u^*(\cdot)}$ defined in \eqref{eq:def_BRSE_ustar} using standard methods \cite[Sec.~3.2]{Althoff2010diss}:
\begin{align}
	\outerBRSE{-\tau;u^*(\cdot)} &= \bigcup_{k \in \{0,...,\steps{}-1\}} \outerBRSE{-\tau_k;u^*(\cdot)} \label{eq:outerBRSE} \\
	\begin{split} \label{eq:outerBRSE_k}
		\outerBRSE{-\tau_k;u^*(\cdot)} &= \convOp{e^{-At_{k+1}} \conZonoOp{\targetset{}}, e^{-At_k} \conZonoOp{\targetset{}}} \\
		&\hspace{15pt} \oplus \F{} e^{-At_{k+1}} \conZonoOp{\targetset{}} \oplus -\outerZW{-\tau_k} \\
		&\hspace{15pt} \oplus -\outerZtraj{{u}}{-\tau_k} ,
	\end{split}
\end{align}
where we use the center input trajectory $\forall t \in \tau\colon u_0(t) = \centerOp{\inputset{}}$ for $u^*(\cdot)$.
Note that the union in \eqref{eq:outerBRSE} is represented implicitly over a number of steps $\steps{}$ with $\tau = \tau_0 \cup ... \cup \tau_{\steps{}-1}$.

Step 2: We enclose other reachable sets $\BRSE{-\tau;u_j(\cdot)}, j \in \Nint{1}{\numInp{}}$ by halfspaces $\polyHRep{\dir{j}^\top, p_{(j)}}$ for an efficient intersection with $\BRSE{-\tau;u_0(\cdot)}$ computed in step 1.
To obtain a small intersection in \eqref{eq:BRSAE_ti}, we maximize the extent of individual $\BRSE{-\tau;u_j(\cdot)}$ toward certain directions; we heuristically choose the $2n$ columns in $[I_n \; -I_n]$.
For each $\BRSE{-\tau;u_j(\cdot)}$, we evaluate the support function in the direction $\dir{j}$ \cite[Sec.~4.1]{Wetzlinger2023HSCC}
\begin{align}
	\begin{split} \label{eq:outerBRSE_sF}
		&\sF{\outerBRSE{-\tau_k;u_j(\cdot)},\dir{j}} \\
		&= \max \big\{ \sF{\targetset{},(e^{-At_k})^\top \dir{j}}, \sF{\targetset{},(e^{-At_{k+1}})^\top \dir{j}} \big\} \\
		&\hspace{11pt} + \sF{\F{} \targetset{}, (e^{-At_{k+1}})^\top \dir{j} } + \sF{\outerZW{-\tau_k},\dir{j}} + \beta_j
	\end{split}
\end{align}
where $\beta_j$ represents the support function of the particular solution due to the input trajectory $u_j(\cdot)$, which is chosen such that the effect of the input set $\inputset{}$ in the direction $\dir{j}$ is minimized:
\begin{equation} \label{eq:innerZtraj}
	\beta_j = \sF{\innerZtraj{u_j}{-t_k},\dir{j}} = - \sF{\innerZU{-t_k},-\dir{j}} .
\end{equation}
Next, we prove that the outlined procedure indeed computes an outer approximation:

\begin{figure}[t]
\centering
\begin{tikzpicture}[scale=1,
	every node/.style={font=\small}]
	
	
	
	\begin{pgfonlayer}{foreground}		
		\draw[goalset] (6.2,3.25) -- (7.2,3.25) -- (7.7,2.25) -- (6.7,1.75) -- (5.7,2.75) -- cycle;
		\node[text=TUMblue] at (6.7,2.6) {$\targetset{}$};
	\end{pgfonlayer}

	\begin{scope}[]
		\clip (0.5,1.59) rectangle (8,4);
		\fill[lightgray] (1.16,2.14) .. controls (3,2.5) and (5,3.75) .. (7.2,3.25)
		-- (6.7,1.75) .. controls (4.5,2.5) and (3,1.5) .. (1.84,0.8) -- (0.94,1.3) -- cycle;
		\node[anchor=west] at (5.75,3.75) {$\outerBRSAE{-\tau}$};
		\draw (5.75,3.75) edge[->,bend right,>=stealth'] (5,3);
	\end{scope}
	
	\draw[black] (1.16,2.14) .. controls (3,2.5) and (5,3.75) .. (7.2,3.25)
	-- (6.7,1.75) .. controls (4.5,2.5) and (3,1.5) .. (1.84,0.8) -- (0.94,1.3) -- cycle;
	\node[anchor=west] at (3,0.7) {$\outerBRSE{-\tau;u_0(\cdot)}$};
	\draw (3,0.7) edge[->,bend left,>=stealth'] (2.25,1);
	
	\draw[draw=black,dashed] (1.26,3.24) .. controls (3,2.5) and (5,3.75) .. (7.2,3.25)
	-- (6.7,1.75) .. controls (4.5,2.5) and (3,1.5) .. (1.4,1.6) -- (0.94,2.4) -- cycle;
	\node[anchor=west] at (2.25,3.75) {$\outerBRSE{-\tau;u_1(\cdot)}$};
	\draw (2.25,3.75) edge[->,bend right,>=stealth'] (1.5,3.2);
	
	\draw[->,>=stealth'] (0.4,1.45) -- node[right,midway] {$\dir{1}$} ++(0,-0.5);
	\draw[red,semithick] (0.25,1.59) -- ++(7.75,0) node[yshift=-0.35cm,xshift=-1.35cm] {$\sF{\outerBRSE{-\tau;u_1(\cdot)},\dir{1}}$};
	
\end{tikzpicture}
\caption{Computation of an outer approximation of the AE backward reachable set $\outerBRSAE{-\tau}$ by intersection of multiple backward reachable sets for specific input trajectories (shown for two trajectories $u_0(\cdot), u_1(\cdot)$):
	The set $\outerBRSE{-\tau;u_0(\cdot)}$ is intersected with the halfspace constructed by the support function of the other set $\outerBRSE{-\tau;u_1(\cdot)}$ in the direction $\dir{1}$, with the input trajectory minimizing the extent of the set in that direction.}
\label{fig:BRSAE_ti}
\end{figure}

\begin{theorem}[Time-interval AE backward reachable set] \label{thm:BRSAE_ti}
Let the subset $\inputsignalssub{} \subset \inputsignals{}$ be composed of $\numInp{} \in \N{}$ input trajectories.
The time-interval AE backward reachable set \eqref{eq:def_BRSAE_ti} can be outer approximated by
\begin{equation} \label{eq:outerBRSAE_ti}
	\outerBRSAE{-\tau} = 
		\bigcup_{k \in \{0,...,\steps{}-1\}} \big( \outerBRSE{-\tau_k;u_0(\cdot)} \cap \polyHRep{N,p} \big)
\end{equation}
where $\outerBRSE{-\tau;u_0(\cdot)}$ is computed by \eqref{eq:outerBRSE} using the center trajectory $u_0(\cdot)$, for which $\forall t \in \tau\colon u(t) = \centerOp{\inputset{}}$ holds, and
\begin{align}
\begin{split} \label{eq:innerPoly}
	&\forall j \in \Nint{1}{\numInp{}}\colon N_{(j,\cdot)} = \dir{j}^\top, \\
	&\forall j \in \Nint{1}{\numInp{}}\colon p_{(j)} = \max_{k \in \{1,...,\steps{}\}} \sF{\innerBRSE{-\tau_k;u_j(\cdot)},\dir{j}}
\end{split}
\end{align}
constructs a polytope via support function evaluations of the outer approximation $\outerBRSE{-\tau;u_j(\cdot)}$ of the backward reachable set \eqref{eq:def_BRSE_ustar} using all $\numInp{}$ input trajectories in $\inputsignalssub{}$.
\end{theorem}
\begin{proof}
See Appendix.
\end{proof}

\cref{alg:BRSAE_ti} implements \cref{thm:BRSAE_ti}:
In the main loop, we iteratively compute an explicit outer approximation $\outerBRSE{-\tau;u^*(\cdot)}$ of time-inverted dynamics $\dot{\tilde{x}}(t) = -A\tilde{x}(t) - Bu^*(t) - Ew(t)$, where we use the center input trajectory $\forall t \in \tau\colon u^*(t) = \centerOp{\inputset{}}$.
Furthermore, we choose the columns in $N = [I_n \; -I_n]$ as the minimizing directions for the other input trajectories $u(\cdot) \in \inputsignals{}$ and propagate the corresponding support functions of the corresponding outer approximations for the time-inverted dynamics (lines~\ref{alg:BRSAE_ti:innerZW_prop}-\ref{alg:BRSAE_ti:polyvec}).
Ultimately, the intersection of the constructed polytope $\polyHRep{N,p}$ with the explicit outer approximation $\outerBRSE{-\tau;u^*(\cdot)}$ (\cref{alg:BRSAE_ti:union}) yields the outer approximation of the time-interval AE backward reachable set $\outerBRSAE{-\tau}$.
Please note that the union is represented implicitly by a sequence of time-interval solutions.

\begin{algorithm}[!t]
\caption{Time-interval AE backward reachable set} \label{alg:BRSAE_ti}
\textbf{Require:} Linear system $\dot{x} = Ax + Bu + Ew$,
target set~$\targetset{} = \polyHRep{\polymat{},\polyvec{}}$, input set~$\inputset{} = \zono{c_u,G_u}$,
disturbance set~$\distset{} = \zono{c_w,G_w}$, time interval~$\tau = [t_0,\tFinal{}]$, steps~$\steps{} \in \N{}$

\textbf{Ensure:} Outer approximation of the time-interval backward reachable set $\outerBRSAE{-\tau}$

\setstretch{1.2}
\begin{algorithmic}[1]
	\State $\Delta t \gets (\tFinal{}-t_0)/\steps{}$, $w \gets \centerOp{\distset{}} + \centerOp{\inputset{}}$
		\label{alg:BRSAE_ti:init}
	\State $\distsetzero{} \gets \zono{\matzeros{},G_w}$, $\inputsetzero{} \gets \zono{\matzeros{},G_u}$%
		\label{alg:BRSAE_ti:shiftWU}
	\State $\F{} \gets$ Eq.~\eqref{eq:F}, $\CZ{} \gets \conZonoOp{\targetset{}}$
		\label{alg:BRSAE_ti:targetsetCZ}
		\textcolor{gray}{\Comment{see~\cref{alg:conZonoOp}}}
	\State $N \gets [I_n \; -\!I_n]$, $\numInp{} \gets 2n$, $\forall j \in \Nint{1}{\numInp{}}\colon p_{(j)} \gets \infty$
		\label{alg:BRSAE_ti:initpoly}
	\State pre-compute $\outerZWzero{-\Delta t}$ and $\outerZWzero{-t_0}$
		\label{alg:BRSAE_ti:outerZW_preprop}
		\textcolor{gray}{\Comment{see~\eqref{eq:outerZ}, \eqref{eq:Zprop}}}
	\State $\forall j \in \Nint{1}{\numInp{}}\colon$pre-compute $\sF{\outerZWzero{-t_0},N_{(\cdot,j)}}$ and
	\Statex \hspace{48pt} $\sF{\innerZUzero{-t_0},-N_{(\cdot,j)}}$
		\textcolor{gray}{\Comment{see~\eqref{eq:linMap_sF}, \eqref{eq:minkSum_sF}, \eqref{eq:Zprop}}}
	\For{$k \gets 0$ to $\steps{}-1$} \label{alg:BRSAE_ti:loopstart}
		\State $t_{k+1} \! \gets t_k + \Delta t$, $\tau_k \gets [t_k,t_{k+1}]$, $\outerZtraj{w}{-\tau_k} \gets$ Eq.~\eqref{eq:outerZtraj}%
			\label{alg:BRSAE_ti:incrementtime}
		\State $\outerZWzero{-t_{k+1}} \gets \outerZWzero{-t_k} \oplus e^{-A t_k} \outerZWzero{-\Delta t}$
			\label{alg:BRSAE_ti:outerZWzero_prop}
		\State $\outerZW{-\tau_k} \gets \outerZWzero{-t_{k+1}} \oplus \outerZtraj{w}{-\tau_k}$%
			\label{alg:BRSAE_ti:outerZW_prop}
		\State $\outerBRSE{-\tau_k} \gets \convOp{e^{-At_{k+1}} \CZ{}, e^{-At_k} \CZ{}}$
			\label{alg:BRSAE_ti:outerBRSE}
		\Statex \hspace{67pt} $\oplus \, \F{} e^{-At_{k+1}} \CZ{} \oplus -\outerZW{-\tau_k}$
		\State $\forall j \in \Nint{1}{\numInp{}}\colon$propagate $\sF{\outerZW{-\tau_k},N_{(\cdot,j)}}$ and
			\label{alg:BRSAE_ti:innerZW_prop}
		\Statex \hspace{63pt} $\sF{\innerZUzero{-t_{k+1}},-N_{(\cdot,j)}}$
			\textcolor{gray}{\Comment{see~\eqref{eq:linMap_sF}, \eqref{eq:minkSum_sF}}}
		\State $\forall j \in \Nint{1}{\numInp{}}\colon \sF{\innerBRSE{-\tau_k},N_{(\cdot,j)}} \gets$ Eq.~\eqref{eq:outerBRSE_sF}
		\State $\forall j \in \Nint{1}{\numInp{}}\colon p_{(j)} \hspace{-1pt} \gets \hspace{-1pt} \max \Big\{ p_{(j)}, \sF{\innerBRSE{-\tau_k},N_{(\cdot,j)}}\hspace{-3pt}\Big\}$%
			\label{alg:BRSAE_ti:polyvec}
	\EndFor \label{alg:BRSAE_ti:loopend}
	\State $\outerBRSAE{-\tau} = \bigcup_{k=0}^{\steps{}-1} \outerBRSE{-\tau_k} \cap \polyHRep{N,p}$ 
		\textcolor{gray}{\Comment{see \eqref{eq:intersection_CZ}}}
		\label{alg:BRSAE_ti:union}
\end{algorithmic}
\end{algorithm}

\subsubsection{Runtime Complexity}
\label{sssec:outerBRSAE_ti_bigO}

The dominating operations are the conversion of the target set $\targetset{}$ to a constrained zonotope in \cref{alg:BRSAE_ti:targetsetCZ} and the intersection in \cref{alg:BRSAE_ti:union}, which are both $\bigO{n^{4.5}}$ according to \cref{tab:setops} and under \cref{ass:runtimecomplexity}.

\subsubsection{Approximation Error}
\label{sssec:outerBRSAE_ti_eps}

The approximation error of the intermediate result $\outerBRSE{-\tau;u^*(\cdot)}$ in \eqref{eq:outerBRSE} converges to $0$ for $\Delta t \to 0$, since the approximation error of particular solution $\outerZW{t}$ converges to $0$ by \cref{prop:convergence}, the error term $\F{} e^{-At_{k+1}} \conZonoOp{\targetset{}}$ converges to $\{\matzeros{}\}$ as $\lim_{\Delta t \to 0} \F{} = [\matzeros{},\matzeros{}]$ \cite[Lemma~3]{Wetzlinger2023TAC}, and all sets are closed under the applied set operations.
For a zero approximation error everywhere, one would have to consider all combinations of directions of the support function of the particular solution $\ZU{t}$ and directions, in which to compute the intersection with $\outerBRSE{-\tau}$.

\subsubsection{Unperturbed Case}
\label{sssec:BRSAE_ti_noW}

Setting $\distset{} = \{\matzeros{}\}$ removes every occurrence of the particular solution $\outerZW{t}$ in \cref{alg:BRSEA_ti}.
Both runtime complexity and approximation error are unchanged.

\section{Maximal Backward Reachability Analysis}
\label{sec:backward_maximal}

In this section, we compute inner and outer approximations of the time-point EA backward reachable set $\BRSEA{-t}$ given in \eqref{eq:def_BRSEA_tp} in \cref{ssec:BRSEA_tp} as well as an inner approximation of the time-interval EA backward reachable set $\BRSEA{-\tau}$ given in \eqref{eq:def_BRSEA_ti} in \cref{ssec:BRSEA_ti}.
We show that the runtime complexity of our algorithms is polynomial in the state dimension $n$, examine the approximation errors, and discuss simplifications for the unperturbed cases $\BRSE{-t}$ and $\BRSE{-\tau}$ defined in \eqref{eq:def_BRSE}.

\subsection{Time-Point Solution}
\label{ssec:BRSEA_tp}

We base the computation of the backward reachable set $\BRSEA{-t}$ on the following proposition:
\begin{proposition}[Time-point EA backward reachable set] \label{prop:BRSEA_tp}
The backward reachable set $\BRSEA{-t}$ defined in \eqref{eq:def_BRSEA_tp} can be computed by
\begin{equation} \label{eq:BRSEA_tp}
	\BRSEA{-t} = e^{-At} \big( ( \targetset{} \ominus \ZW{t} ) \oplus -\ZU{t} \big) .
\end{equation}
\end{proposition}
\begin{proof}
See Appendix.
\end{proof}
\noindent
The formula \eqref{eq:BRSEA_tp} above holds independently of the chosen set representations.
Next, we compute approximations in polynomial time assuming a polytopic target set $\targetset{}$ and zonotopic particular solutions $\ZW{t}$ and $\ZU{t}$.

\subsubsection{Outer and Inner Approximation}

We evaluate the Minkowski difference $\targetset{} \ominus \ZW{t}$ using \eqref{eq:minkDiff_poly} and convert the resulting polytope $\targetset{} \ominus \ZW{t}$ to a constrained zonotope using \cref{alg:conZonoOp} for the Minkowski sum with the zonotope $-\ZU{t}$.
For an outer approximation, we underestimate the influence of the disturbance by $\innerZW{t} \subseteq \ZW{t}$ and overestimate the influence of the control input by $\outerZU{t} \supseteq \ZU{t}$:
\begin{align}
	&\BRSEA{-t} = e^{-At} \big( ( \targetset{} \ominus \ZW{t} ) \oplus -\ZU{t} \big) \nonumber \\
	&\overset{\eqref{eq:outerZ},\,\eqref{eq:innerZ}}{\subseteq} \! e^{-At} \big( \conZonoOp{ \targetset{} \ominus \innerZW{t} } \oplus -\outerZU{t} \big) \eqqcolon \outerBRSEA{-t} \label{eq:outerBRSEA_tp}
\end{align}
and vice versa to compute an inner approximation:
\begin{align}
	&\BRSEA{-t} = e^{-At} \big( ( \targetset{} \ominus \ZW{t} ) \oplus -\ZU{t} \big) \nonumber \\
	&\overset{\eqref{eq:outerZ},\,\eqref{eq:innerZ}}{\supseteq} \! e^{-At} \big( \conZonoOp{ \targetset{} \ominus \outerZW{t} } \oplus -\innerZU{t} \big) \eqqcolon \innerBRSEA{-t} . \label{eq:innerBRSEA_tp}
\end{align}

\subsubsection{Runtime Complexity}
\label{sssec:BRSEA_tp_bigO}

Under \cref{ass:runtimecomplexity}, the dominating operation in \eqref{eq:outerBRSEA_tp} and \eqref{eq:innerBRSEA_tp} is the conversion to a constrained zonotope, which is $\bigO{n^{4.5}}$, as the Minkowski sums, Minkowski differences, and linear maps are $\bigO{n^3}$.

\subsubsection{Approximation Error}
\label{sssec:BRSEA_tp_eps}

The sets are closed under the applied operations, so that the entire approximation error is incurred by the outer and inner approximations of the particular solutions, which converges to $0$ as $\Delta t \to 0$ by \cref{prop:convergence}.
Hence, the approximation errors of $\outerBRSEA{-t}$ and $\innerBRSEA{-t}$ also approach $0$ as $\Delta t \to 0$.

\subsubsection{Unperturbed Case}
\label{sssec:BRSEA_tp_noW}

For $\distset{} = \{\matzeros{}\}$, both approximations in \eqref{eq:outerBRSEA_tp}-\eqref{eq:innerBRSEA_tp} simplify accordingly, yielding $\BRSE{-t}$, see \eqref{eq:def_BRSE} with $\tau = t$, with the same runtime complexity and behavior of the approximation error in the limit $\Delta t \to 0$ as above.

\subsection{Time-Interval Solution}
\label{ssec:BRSEA_ti}

For the time-interval solution $\BRSEA{-\tau}$ as defined in \eqref{eq:def_BRSEA_ti}, we want to compute an inner approximation so that all states are guaranteed to reach the target set $\targetset{}$.
Our main idea is to inner approximate the union of time-point solutions $\bigcup_{t \in \tau} \BRSEA{-t}$, which by \cref{prop:union} is an inner approximation of the time-interval solution $\BRSEA{-\tau}$.
We now show how to compute this inner approximation in polynomial time.

\subsubsection{Inner Approximation}
\label{sssec:innerBRSEA_ti}

We require the following lemma: 
\begin{lemma}[Distributivity of Minkowski difference over convex hull] \label{lmm:convMinkDiff}
For three compact, convex, and nonempty sets $\S{1}, \S{2}, \S{3} \subset \R{n}$, we have
\begin{equation*}
	\convOp{\S{1} \ominus \S{3},\S{2} \ominus \S{3}} \subseteq \convOp{\S{1},\S{2}} \ominus \S{3} .
\end{equation*}
\end{lemma}
\begin{proof}
See Appendix.
\end{proof}
%
Next, we exploit the superposition principle to inner approximate the union of time-point solutions over a time interval $\tau$:
\begin{theorem}[Time-interval EA backward reachable set] \label{thm:BRSEA_ti}
The union of time-point EA backward reachable sets
\begin{equation} \label{eq:BRSEA_tp_union}
	\bigcup_{t \in \tau_k} \BRSEA{-t}
	= \big \{ e^{-At} \big( ( \targetset{} \ominus \ZW{t} ) \oplus -\ZU{t} \big) \, \big| \, t \in \tau_k \big \}
\end{equation}
over $\tau_k = [t_k,t_{k+1}]$ can be inner approximated by
\begin{align}
\begin{split} \label{eq:innerBRSEA_ti}
	&\innerBRSEA{-\tau_k} = e^{-A t_{k+1}} \big( -\innerZU{t_k} \, \oplus \\
	&\hspace{5pt} \operator{conv} \big( \conZonoOp{ ((\targetset{} \ominus \F{} \, \boxOp{\targetset{}}) \ominus \B{\mu}) \ominus \outerZW{\tau_k} } , \\
	&\hspace{5pt} \conZonoOp{ ((e^{A \Delta t} \targetset{} \ominus \F{} \, \boxOp{\targetset{}}) \ominus \B{\mu}) \ominus \outerZW{\tau_k} } \big) \big) ,
\end{split}
\end{align}
where all variables are computed as introduced in \cref{ssec:FRS}.
The union over all $\steps{}$ steps, that is,
\begin{equation*}
	\innerBRSEA{-\tau} = \bigcup_{k \in \{0,...,\steps{}-1\}} \innerBRSEA{-\tau_k} ,
\end{equation*}
is an inner approximation of the time-interval backward reachable set $\BRSEA{-\tau}$ in \eqref{eq:def_BRSEA_ti} over the time interval $\tau = [t_0,\tFinal{}]$.
\end{theorem}
\begin{proof}
See Appendix.
\end{proof}


\begin{algorithm}[!t]
	\caption{Time-interval EA backward reachable set} \label{alg:BRSEA_ti}
	\textbf{Require:} Linear system $\dot{x} = Ax + Bu + Ew$,
	target set~$\targetset{} = \polyHRep{\polymat{},\polyvec{}}$, input set~$\inputset{} = \zono{c_u,G_u}$,
	disturbance set~$\distset{} = \zono{c_w,G_w}$, time interval~$\tau{} = [t_0,\tFinal{}]$, steps~$\steps{} \in \N{}$
	
	\textbf{Ensure:} Inner approximation of the time-interval backward reachable set $\innerBRSEA{-\tau}$
	
	\setstretch{1.2}
	\begin{algorithmic}[1]
		\State $\Delta t \gets (\tFinal{}-t_0)/\steps{}$, $w \gets \centerOp{\distset{}}$, $\distsetzero{} \gets \zono{\matzeros{},G_w}$
		\label{alg:BRSEA_ti:time}
		\State pre-compute $\innerZU{t_0}$ and $\outerZWzero{t_0}$
		\label{alg:BRSEA_ti:initZUZW}
		\Comment{see~\eqref{eq:outerZ}, \eqref{eq:innerZ}, \eqref{eq:Zprop}}
		%
		%
		\State $\mu \gets \sqrt{\gamma} \, \norm{(e^{A\Delta t} - I_n) G}_2$
		\label{alg:BRSEA_ti:mu}
		\textcolor{gray}{\Comment{$G$ and $\gamma$ from $\boxOp{\targetset{}}$}}
		\State $\poly{1} \gets (\targetset{} \ominus \F{} \, \boxOp{\targetset{}}) \ominus \B{\mu}$
		\label{alg:BRSEA_ti:P1}
		\textcolor{gray}{\Comment{see~\eqref{eq:F}, \eqref{eq:innerHti}}}
		\State $\poly{2} \gets (e^{A\Delta t} \targetset{} \ominus \F{} \, \boxOp{\targetset{}}) \ominus \B{\mu}$
		\label{alg:BRSEA_ti:P2}
		\textcolor{gray}{\Comment{see~\eqref{eq:F}, \eqref{eq:innerHti}}}
		\For{$k \gets 0$ to $\steps{}-1$} \label{alg:BRSEA_ti:loopstart}
			\State $t_{k+1} \gets t_k + \Delta t$, $\tau_k \gets [t_k,t_{k+1}]$
			\label{alg:BRSEA_ti:incrementtime}
			\State $\innerZU{t_{k+1}} \gets \innerZU{t_k} \oplus e^{At_k} \innerZU{\Delta t}$
			\label{alg:BRSEA_ti:innerZU_prop}
			\State $\outerZWzero{t_{k+1}} \gets \outerZWzero{t_k} \oplus e^{At_k} \outerZWzero{\Delta t}$
			\label{alg:BRSEA_ti:outerZWzero_prop}
			\State $\outerZtraj{w}{\tau_k} \gets$ Eq.~\eqref{eq:outerZtraj}, $\outerZW{\tau_k} \gets \outerZWzero{t_{k+1}} \oplus \outerZtraj{w}{\tau_k}$%
			\label{alg:BRSEA_ti:outerZW_prop}
			\State $\CZ{} \gets \convOp{ \big( \conZonoOp{\poly{1} \ominus \outerZW{\tau_k}}, \conZonoOp{\poly{2} \ominus \outerZW{\tau_k}} }$%
			\label{alg:BRSEA_ti:innerHti}
			\State $\innerBRSEA{-\tau_k} \gets e^{-A t_{k+1}} (\CZ{} \oplus -\innerZU{t_k})$
			\label{alg:BRSEA_ti:innerBRSEA}
		\EndFor \label{alg:BRSEA_ti:loopend}
		\State $\innerBRSEA{-\tau} = \bigcup_{k=0}^{\steps{}-1} \innerBRSEA{-\tau_k}$
		\label{alg:BRSEA_ti:union}
	\end{algorithmic}
\end{algorithm}

\cref{alg:BRSEA_ti} implements \cref{thm:BRSEA_ti}, where we explicitly consider the more general case of a time interval $\tau = [t_0,\tFinal{}]$ with $t_0 > 0$:
We pre-compute the particular solutions $\innerZU{t}$ and $\outerZW{t}$ until time $t_0$ in line~\ref{alg:BRSEA_ti:initZUZW} and pre-compute the polytopes $\poly{1}, \poly{2}$ (lines~\ref{alg:BRSEA_ti:P1}-\ref{alg:BRSEA_ti:P2}) that are used for inner approximating the time-interval homogeneous solution, see \eqref{eq:innerHti}.
The main loop computes all individual backward reachable sets $\innerBRSEA{-\tau_k}$ following \cref{thm:BRSEA_ti}, which implicitly represent the union (\cref{alg:BRSEA_ti:union}) that is the inner approximation of the time-interval EA backward reachable set $\innerBRSEA{-\tau}$.

\subsubsection{Runtime Complexity}
\label{sssec:innerBRSEA_ti_bigO}

\hspace{-2.5pt}\footnote{This subsection has been altered with respect to the published version.} Under \cref{ass:runtimecomplexity} and following \cref{tab:setops}, only the operation $\boxOp{\targetset{}}$ is $\bigO{n^{4.5}}$, as we can remove all other linear programs from \cref{alg:BRSEA_ti}, which occur in the exact conversion operation $\conZonoOp{\poly{}}$ in \cref{alg:BRSEA_ti:innerHti}.
According to \cite[Thm.~3]{Scott2016Automatica}, \cref{alg:conZonoOp} works with \emph{any} enclosure of $\poly{}$.
Hence, we can use the pre-computed set $\boxOp{\targetset{}}$ in all steps as
\begin{equation*}
	\forall t \in \tau, \forall i \in \{1,2\}\colon \poly{i} \ominus \outerZW{t} \subseteq \boxOp{\targetset{}} .
\end{equation*}
As a consequence, increasing the number of steps $\steps{}$ and thereby improving the tightness is only $\bigO{n^3}$.

\subsubsection{Approximation Error}
\label{sssec:innerBRSEA_ti_eps}

By \cref{prop:convergence}, the approximation error of the particular solutions $\outerZW{t_{k+1}}$ and $\innerZU{t_k}$ converges to $0$ as $\Delta t \to 0$.
Moreover, the sets $\poly{1}$ and $\poly{2}$ converge to $\targetset{}$ as $\lim_{\Delta t \to 0} \F{} = [\matzeros{},\matzeros{}]$ by \cite[Lemma~1]{Wetzlinger2023TAC} and $\lim_{\Delta t \to 0} \mu \overset{\eqref{eq:mu}}{=} 0$.
Consequently, the computed individual time-interval solutions $\innerBRSEA{-\tau_k}$ converge to the exact time-point solution $\BRSEA{-t_k}$ in the limit $\Delta t \to 0$.
However, a non-zero approximation error remains even in the limit as the union of time-point solutions is an inner approximation of the time-interval solution by \cref{prop:union}.

\subsubsection{Unperturbed Case}
\label{sssec:BRSEA_ti_noW}

As mentioned in \cref{ssec:maximal}, the unperturbed case is equivalent to computing the forward reachable set as defined in \cref{def:FRS} for the time-inverted dynamics $\dot{x}(t) = -Ax(t) - Bu(t)$.
For $\distset{} = \{\matzeros{}\}$, \cref{alg:BRSEA_ti} simplifies to computing an inner approximation of this forward reachable set in $\bigO{n^{4.5}}$.
Consequently, the approximation error converges to $0$ in the limit $\Delta t \to 0$ \cite[Thm.~1]{Wetzlinger2023TAC}.

\section{Numerical Examples}
\label{sec:numericalexamples}

We implemented our algorithms using the MATLAB toolbox CORA \cite{Althoff2015ARCH} for set-based computing and MOSEK\footnote{Available at \url{https://www.mosek.com}.} for solving linear programs. 
All computations are carried out on a 2.60GHz six-core i7 processor with 32GB RAM.

\begin{table}[t]
	\centering \small
	\caption{Results of \cref{ssec:pursuitevasion,ssec:groundcollision,ssec:terminalset}.}
	\label{tab:allresults}
	\begin{tabular}{l l l}
		\toprule
		\textbf{Benchmark} & \textbf{Algorithm} & \textbf{Time} \\ \midrule
		\multirow{3}{*}{\cref{ssec:pursuitevasion}: $\outerBRSAE{-\tau}$} & Alg.\ 2 ($\steps = 100$) & $0.11\si{\second}$ \\
		& HJ ($\grid{} = 15$) & $2.4\si{\second}$ \\
		& HJ ($\grid{} = 35$) & $197\si{\second}$ \\ \midrule
		\multirow{3}{*}{\cref{ssec:pursuitevasion}: $\innerBRSEA{-\tau}$} & Alg.\ 3 ($\steps = 100$) & $0.12\si{\second}$ \\
		& HJ ($\grid{} = 15$) & $2.4\si{\second}$ \\
		& HJ ($\grid{} = 35$) & $194\si{\second}$ \\ \midrule
		\cref{ssec:groundcollision}: $\outerBRSAEsuper{-\tau}{1,2,3}$ & Alg.\ 2 ($\steps = 200$) & $2.4\si{\second}$ \\ \midrule
		\cref{ssec:terminalset}: $\innerBRSEAsuper{-\tau}{1,2,3}$ & Alg.\ 3 ($\steps = 1000$) & $6.3\si{\second}$ \\
		\bottomrule
	\end{tabular}
\end{table}








\subsection{Pursuit-Evasion Game}
\label{ssec:pursuitevasion}

First, we compare the results with the Python implementation\footnote{Available at \href{https://github.com/StanfordASL/hj_reachability}{https://github.com/StanfordASL/hj\_reachability}.} of the state-of-the-art Hamilton-Jacobi reachability analysis \cite{Chen2018ANNUREV} on a 4D pursuit-evasion game defined by the double integrator dynamics \cite[Eq.~(24)]{Chen2015CDC}
\begin{equation*} 
	A =
	\begin{bmatrix}
		0 & 1 & 0 & 0 \\
		0 & 0 & 0 & 0 \\
		0 & 0 & 0 & 1 \\
		0 & 0 & 0 & 0
	\end{bmatrix}, \;
	B =
	\begin{bmatrix}
		0 & 0 \\ 1 & 0 \\ 0 & 0 \\ 0 & 1
	\end{bmatrix}, \;
	E =
	\begin{bmatrix}
		0 & 0 \\ -1 & 0 \\ 0 & 0 \\ 0 & -1
	\end{bmatrix}.
\end{equation*}
%
The state is comprised of the relative positions and velocities in the horizontal and vertical plane, while the control inputs and disturbances represent the corresponding accelerations of Player 1 and Player 2, respectively.
We choose
\begin{align*}
	\targetset{} &= [-0.5,0.5] \times \dotsc \times [-0.5,0.5] \subset \R{4} \\
	\inputset{} &= [-0.5,0.1] \times [-0.1,0.5] \subset \R{2} \\
	\distset{} &= [-0.1,0.5] \times [-0.5,0.1] \subset \R{2}
\end{align*}
where $\targetset{}$ defines a collision between the players, and $\inputset{}$ and $\distset{}$ are chosen such that each player has different steering capacities.
Furthermore, we set the time horizon to $\tau = [0,1]$.
%
%
%
%

\begin{figure}
	\definecolor{mycolor1}{rgb}{0.00000,0.36078,0.67059}%
\begin{tikzpicture}
\pgfplotsset{
plotstyle1/.style={area legend, color=mycolor1, line width=1.25pt},
plotstyle2/.style={area legend, draw=black},
plotstyleHJ15/.style={color=goodRed, only marks, mark size=0.5pt, mark=*, mark options={solid, goodRed}, forget plot},
plotstyleHJ35/.style={color=goodYellow, only marks, mark size=0.5pt, mark=*, mark options={solid, goodYellow}, forget plot},
every tick label/.append style={font=\scriptsize},
/.style={}
}
\def\rows{1}
\def\cols{2}
\def\horzsep{1.25cm}
\def\basepath{figures/}

\begin{groupplot}[%
group style={rows = \rows, columns = \cols, horizontal sep = \horzsep},
scale only axis,
width=1/\cols*0.5*\textwidth -\horzsep,
xlabel style={at={(0.5,-0.1)}},
ylabel style={at={(-0.1,0.5)}},
legend style={legend columns=2,legend to name=legendname,font=\scriptsize,legend cell align=left,/tikz/every even column/.append style={column sep=0.5cm}}
]
\nextgroupplot[xmin=-1.2500,xmax=1.2500,ymin=-1.2500,ymax=1.2500,xlabel={$x_1$},ylabel={$x_2$}]
\input{\basepath pursuitevasion_minimal_11.tikz}
\coordinate (top) at (rel axis cs:0,1);
\nextgroupplot[xmin=-1.2500,xmax=1.2500,ymin=-1.2500,ymax=1.2500,xlabel={$x_3$},ylabel={$x_4$}]
\input{\basepath pursuitevasion_minimal_legends.tikz}
\input{\basepath pursuitevasion_minimal_12.tikz}
\coordinate (bot) at (rel axis cs:1,0);
\end{groupplot}
\path (top|-current bounding box.south)--coordinate(legendpos)(bot|-current bounding box.south);
\node at([yshift=-3ex,xshift=-1ex]legendpos) {\pgfplotslegendfromname{legendname}};

\end{tikzpicture}%
	\caption{Projections of the time-interval AE backward reachable set for the pursuit-evasion game in \cref{ssec:pursuitevasion}.}
	\label{fig:pursuitevasion_minimal}
\end{figure}

\begin{figure}
	\definecolor{mycolor1}{rgb}{0.00000,0.36078,0.67059}%
\begin{tikzpicture}
\pgfplotsset{
plotstyle1/.style={area legend, color=mycolor1, line width=1.25pt},
plotstyle2/.style={area legend, draw=black},
plotstyleHJ15/.style={color=goodRed, only marks, mark size=0.5pt, mark=*, mark options={solid, goodRed}, forget plot},
plotstyleHJ35/.style={color=goodYellow, only marks, mark size=0.5pt, mark=*, mark options={solid, goodYellow}, forget plot},
every tick label/.append style={font=\scriptsize},
/.style={}
}
\def\rows{1}
\def\cols{2}
\def\horzsep{1.25cm}
\def\basepath{figures/}

\begin{groupplot}[%
group style={rows = \rows, columns = \cols, horizontal sep = \horzsep},
scale only axis,
width=1/\cols*0.5*\textwidth -\horzsep,
xlabel style={at={(0.5,-0.1)}},
ylabel style={at={(-0.1,0.5)}},
legend style={legend columns=2,legend to name=legendname,font=\scriptsize,legend cell align=left,/tikz/every even column/.append style={column sep=0.5cm}}
]
\nextgroupplot[xmin=-1.2500,xmax=1.2500,ymin=-1.2500,ymax=1.2500,xlabel={$x_1$},ylabel={$x_2$}]
\input{\basepath pursuitevasion_maximal_11.tikz}
\coordinate (top) at (rel axis cs:0,1);
\nextgroupplot[xmin=-1.2500,xmax=1.2500,ymin=-1.2500,ymax=1.2500,xlabel={$x_3$},ylabel={$x_4$}]
\input{\basepath pursuitevasion_maximal_legends.tikz}
\input{\basepath pursuitevasion_maximal_12.tikz}
\coordinate (bot) at (rel axis cs:1,0);
\end{groupplot}
\path (top|-current bounding box.south)--coordinate(legendpos)(bot|-current bounding box.south);
\node at([yshift=-3ex,xshift=-1ex]legendpos) {\pgfplotslegendfromname{legendname}};

\end{tikzpicture}%
	\caption{Projections of the time-interval EA backward reachable set for the pursuit-evasion game in \cref{ssec:pursuitevasion}.}
	\label{fig:pursuitevasion_maximal}
\end{figure}


\cref{fig:pursuitevasion_minimal,fig:pursuitevasion_maximal} show projections of the AE and EA backward reachable sets $\outerBRSAE{-\tau}$ and $\innerBRSEA{-\tau}$, respectively, computed by \cref{alg:BRSAE_ti,alg:BRSEA_ti}.
For comparison, we also plot the value function obtained by HJ reachability using $\grid{} \in \{15,35\}$ grid points per dimension over a domain of $[-1.5,1.5]$.
Note that we plot only the grid points with a negative value function to represent $\innerBRSEA{-\tau}$; for $\outerBRSAE{-\tau}$, we plot all grid points with a negative value function evaluation as well as their neighbors in all directions (also diagonally) with nonnegative values.
The plotted grid points indicate that the outer approximation $\outerBRSAE{-\tau}$ tightens with finer sampling, while the inner approximation $\innerBRSEA{-\tau}$ widens.

Our proposed algorithms yield similar\footnote{Slight deviations originate in part from the differences between \cref{def:BRSAE,def:BRSEA} and the definitions used by HJ reachability, as discussed in \cref{fn:defs} on \Cpageref{fn:defs}.} results compared to HJ reachability.
While the runtime complexity of our proposed algorithms only scales linearly with the number of time steps, the computation time of HJ reachability strongly depends on the partitioning on the grid, see \cref{tab:allresults}, as it suffers from the curse of dimensionality.
Furthermore, the grid must cover the domain of the backward reachable set, which ultimately requires knowledge about the solution before computing it.
This is not the case for our proposed backward reachability algorithms.

\subsection{Ground Collision Avoidance}
\label{ssec:groundcollision}

\begin{figure*}
	\definecolor{mycolor1}{rgb}{0.00000,0.36078,0.67059}%
\definecolor{mycolor2}{rgb}{0.89020,0.10588,0.13725}%
\definecolor{mycolor3}{rgb}{1.00000,0.76471,0.14510}%
\begin{tikzpicture}
\pgfplotsset{
plotstyle1/.style={area legend, color=mycolor1, line width=1.2pt},
plotstyle2/.style={area legend, color=mycolor2, dashed, line width=1.2pt},
plotstyle3/.style={area legend, color=mycolor3, dashdotted, line width=1.2pt},
plotstyle4/.style={area legend, draw=black},
every tick label/.append style={font=\scriptsize},
/.style={}
}
\def\rows{1}
\def\cols{3}
\def\horzsep{1.5cm}
\def\basepath{figures/}
\def\subfigheight{3.5cm}

\begin{groupplot}[%
group style={rows = \rows, columns = \cols, horizontal sep = \horzsep},
scale only axis,
width=1/\cols*\textwidth -\horzsep,
legend style={legend columns=4,legend to name=legendname,font=\scriptsize,legend cell align=left,/tikz/every even column/.append style={column sep=0.5cm}}
]
\nextgroupplot[xmin=-2.75,xmax=2.75,ymin=-22,ymax=22,height=\subfigheight,xlabel={$x_1$},ylabel={$x_3$},ylabel style={at={(-0.125,0.5)}}]
\input{\basepath groundcollision_new_11.tikz}
\coordinate (top) at (rel axis cs:0,1);
\nextgroupplot[xmin=-0.25,xmax=2,ymin=-6,ymax=0.5,height=\subfigheight,xlabel={$x_2$},ylabel={$x_4$}]
\input{\basepath groundcollision_new_12.tikz}
\nextgroupplot[xmin=-3,xmax=3,ymin=-0.20,ymax=2.00,height=\subfigheight,xlabel={$x_1$},ylabel={$x_2$},
	yticklabel style={/pgf/number format/fixed,/pgf/number format/precision=2}]
\input{\basepath groundcollision_new_legends.tikz}
\input{\basepath groundcollision_new_13.tikz}
\coordinate (bot) at (rel axis cs:1,0);
\end{groupplot}
\path (top|-current bounding box.south)--coordinate(legendpos)(bot|-current bounding box.south);
\node at([yshift=-3ex]legendpos) {\pgfplotslegendfromname{legendname}};

\end{tikzpicture}%
	\caption{Projections of the time-interval AE backward reachable set for the ground collision avoidance scenario in \cref{ssec:groundcollision}.}
	\label{fig:groundcollision}
\end{figure*}

Next, we examine the computation of the AE backward reachable set
using a linearized longitudinal model of a quadrotor \cite[Eq.~(42)]{Mitchell2019arXiv}
\begingroup\setlength{\arraycolsep}{3pt}
\begin{equation*}
	A =
	\begin{bmatrix}
		0 & 0 & 1 & 0 & 0 & 0 \\
		0 & 0 & 0 & 1 & 0 & 0 \\
		0 & 0 & 0 & 0 & g & 0 \\
		0 & 0 & 0 & 0 & 0 & 0 \\
		0 & 0 & 0 & 0 & 0 & 1 \\
		0 & 0 & 0 & 0 & -d_0 & -d_1
	\end{bmatrix},
	B = 
	\begin{bmatrix}
		0 & 0 \\
		0 & 0 \\
		0 & 0 \\
		K & 0 \\
		0 & 0 \\
		0 & n_0
	\end{bmatrix},
	E = 
	\begin{bmatrix}
		0 & 0 \\
		0 & 0 \\
		1 & 0 \\
		0 & 1 \\
		0 & 0 \\
		0 & 0
	\end{bmatrix}
\end{equation*}
\endgroup
with $g = 9.81, d_0 = 70, d_1 = 17, K = 0.89/1.4$, and $n_0 = 55$.
In order, the states represent the horizontal position, vertical position, horizontal velocity, vertical velocity, roll, and roll velocity.
For our ground collision avoidance scenario, we want to avoid any state $x_2 \leq 0.1$ with a negative velocity $x_4 \leq 0$.
Inspired by \cite[Sec.~6.1]{Mitchell2019arXiv}, we define the target set $\targetset{} = \polyHRep{\polymat{},\polyvec{}} \subset \R{6}$ with
%
\begingroup
\setlength{\arraycolsep}{2pt}
\setcounter{MaxMatrixCols}{20}
\begin{align*}
	\polymat{}^\top &= 
	\begin{bmatrix}
		1 & -1 			 & 0 & \phantom{-}0 &\phantom{-}1  & -1 		  & 0 & \phantom{-}0 & 0 & \phantom{-}0 & \phantom{-}0  & 0 & \phantom{-}0  & 0 &  \phantom{-}0 \\
		0 & \phantom{-}0 & 1 & -1 			& -2 		   & -2 		  & 0 & \phantom{-}0 & 0 & \phantom{-}0 & \phantom{-}10 & 0 & \phantom{-}0  & 0 &  \phantom{-}0 \\
		0 & \phantom{-}0 & 0 & \phantom{-}0 & \phantom{-}0 & \phantom{-}0 & 1 & -1 			 & 0 & \phantom{-}0 & \phantom{-}0  & 0 & \phantom{-}0  & 0 &  \phantom{-}0 \\
		0 & \phantom{-}0 & 0 & \phantom{-}0 & \phantom{-}0 & \phantom{-}0 & 0 & \phantom{-}0 & 1 & -1 			& -1			& 0 & \phantom{-}0  & 0 &  \phantom{-}0 \\
		0 & \phantom{-}0 & 0 & \phantom{-}0 & \phantom{-}0 & \phantom{-}0 & 0 & \phantom{-}0 & 0 & \phantom{-}0 & \phantom{-}0  & 1 & -1 			& 0 &  \phantom{-}0 \\
		0 & \phantom{-}0 & 0 & \phantom{-}0 & \phantom{-}0 & \phantom{-}0 & 0 & \phantom{-}0 & 0 & \phantom{-}0 & \phantom{-}0  & 0 & \phantom{-}0  & 1 & -1  
	\end{bmatrix}, \\
	\polyvec{}^\top &=
	\begin{bmatrix}
		  \tfrac{1}{2}		
		& \tfrac{1}{2}		
		& \tfrac{1}{10}		
		& 0					
		& \tfrac{3}{10}
		& \tfrac{3}{10}
		& 1					
		& 1					
		& 0					
		& 1					
		& 1
		& \tfrac{\pi}{15}	
		& \tfrac{\pi}{15}	
		& \tfrac{\pi}{2}	
		& \tfrac{\pi}{2}	
	\end{bmatrix} .
\end{align*}
\endgroup%
The control inputs are the total normalized thrust and the desired roll angle, while the disturbances capture linearization errors.
Inspired by \cite[Eq.~(45)]{Mitchell2019arXiv}, we bound these values by
{\setlength{\arraycolsep}{2pt}
\begin{align*}
	\inputset{} &= \zono[\big]{\begin{bmatrix} \tfrac{g}{K} & 0 \end{bmatrix}^\top\!\!, \diag \begin{bmatrix} \zeta \tfrac{3}{2} & \tfrac{\pi}{6} \end{bmatrix}} \subset \R{2} , \\
	\distset{} &= \zono[\big]{\begin{bmatrix} 0 & 0 \end{bmatrix}^\top\!\!, \diag \begin{bmatrix} 0.2760 \varphi & 0.3668 \end{bmatrix}} \subset \R{2} ,
\end{align*}}%
%
where the scaling factors $\zeta \in \R{}$ and $\varphi \in \R{}$ allow us to design cases with different input and disturbance capacities, for which we use the following pairs:
$\zeta^{(1)} = 1$ and $\varphi^{(1)} = 10$,
$\zeta^{(2)} = 1$ and $\varphi^{(2)} = 1$, and
$\zeta^{(3)} = 2$ and $\varphi^{(3)} = 1$.
We set $\tau = [0,0.5]$.

\cref{fig:groundcollision} shows the time-interval AE backward reachable sets $\outerBRSAE{-\tau}$ corresponding to the different values of $\zeta$ and $\varphi$, with the computation times in \cref{tab:allresults}.
As expected, the projections show that $\outerBRSAEsuper{-\tau}{1} \supset \outerBRSAEsuper{-\tau}{2} \supset \outerBRSAEsuper{-\tau}{3}$ since the input capacity increases, as $\zeta^{(1)} \leq \zeta^{(2)} \leq \zeta^{(3)}$, and the disturbance capacity decreases, as $\varphi^{(1)} \geq \varphi^{(2)} \geq \varphi^{(3)}$.
In the leftmost projection, we see that $\outerBRSAEsuper{-\tau}{1}$ extends furthest in $\pm x_1$ and $\pm x_3$ because the disturbance $w_1$ is larger than in the other cases and forces more states to enter the target set.
The projections of $\outerBRSAEsuper{-\tau}{2}$ and $\outerBRSAEsuper{-\tau}{3}$ are identical because the input $u_1$ neither directly nor indirectly influences these dimensions.
As indicated by the middle and rightmost plots, an increase of the input capacity of $u_1$ for $\outerBRSAEsuper{-\tau}{3}$ allows more states to avoid the target set $\targetset{}$ in comparison to $\outerBRSAEsuper{-\tau}{2}$, which is affected by the same disturbance set.
Moreover, the middle plot shows that all states with positive vertical velocity $x_4$ can avoid the target set.


\subsection{Terminal Set Reachability}
\label{ssec:terminalset}

In this subsection, we analyze the computation of the EA backward reachable set using a $12$-dimensional quadrotor system linearized about the hover condition \cite[Sec.~2]{Kaynama2014ARCH}.
The state matrix $A \in \R{12 \times 12}$ and the input matrix $B \in \R{12 \times 4}$ are provided by \cite[Appendix~A]{Kaynama2014ARCH}, while the disturbance matrix $E \in \R{12 \times 3}$ is all-zero except for $E_{(4,1)}=E_{(5,2)}=E_{(6,3)}=1$ as in \cite[Sec.~V-D]{Gruber2023TAC}.
To highlight the relation of maximal backward reachability with controller synthesis, we choose a \emph{safe terminal set} \cite[Sec.~IV-A]{Gruber2021CSL} as our target set:
For each state in the safe terminal set, there exists a stabilizing controller keeping the state in the safe terminal set at the next time step and, by induction, for all times.
Our EA backward reachable set contains all states that can be steered into the safe terminal set despite worst-case disturbances.

\begin{figure*}
	\definecolor{mycolor1}{rgb}{0.00000,0.36078,0.67059}%
\definecolor{mycolor2}{rgb}{0.89020,0.10588,0.13725}%
\definecolor{mycolor3}{rgb}{1.00000,0.76471,0.14510}%
\begin{tikzpicture}
\pgfplotsset{
plotstyle1/.style={area legend, color=mycolor1, line width=1pt},
plotstyle2/.style={area legend, color=mycolor2, dashed, line width=1pt},
plotstyle3/.style={area legend, color=mycolor3, dashdotted, line width=1pt},
plotstyle4/.style={area legend, draw=black, fill=white},
every tick label/.append style={font=\scriptsize},
/.style={}
}
\def\rows{1}
\def\cols{4}
\def\horzsep{1.25cm}
\def\vertsep{1.15cm}
\def\basepath{figures/}
\def\subfigheight{3.5cm}

\begin{groupplot}[%
group style={rows = \rows, columns = \cols, horizontal sep = \horzsep, vertical sep = \vertsep},
scale only axis,
ylabel style={at={(-0.1,0.5)}},
width=1/\cols*\textwidth - \horzsep,
legend style={legend columns=4,legend to name=legendname, font=\scriptsize, legend cell align=left,/tikz/every even column/.append style={column sep=0.5cm}}
]
\nextgroupplot[xmin=-0.85,xmax=0.85,ymin=-2.250,ymax=2.250,height=\subfigheight,xlabel={$x_1$},ylabel={$x_4$}]
\input{\basepath terminalSet_11.tikz}
\coordinate (top) at (rel axis cs:-0.05,1);
\nextgroupplot[xmin=-0.850,xmax=0.850,ymin=-2.250,ymax=2.250,height=\subfigheight,xlabel={$x_2$},ylabel={$x_5$}]
\input{\basepath terminalSet_21.tikz}
\nextgroupplot[xmin=-6.000,xmax=2.000,ymin=-5.000,ymax=11.000,height=\subfigheight,xlabel={$x_3$},ylabel={$x_6$}]
\input{\basepath terminalSet_12.tikz}
\nextgroupplot[xmin=-0.60,xmax=0.60,ymin=-1.10,ymax=1.10,height=\subfigheight,xlabel={$x_9$},ylabel={$x_{12}$}]
\input{\basepath terminalSet_legends.tikz}
\input{\basepath terminalSet_22.tikz}
\coordinate (bot) at (rel axis cs:1,0);
\end{groupplot}
\path (top|-current bounding box.south)--coordinate(legendpos)(bot|-current bounding box.south);
\node at([yshift=-3ex]legendpos) {\pgfplotslegendfromname{legendname}};

\end{tikzpicture}%
	\caption{Projections of the time-interval EA backward reachable set for the quadrotor system in \cref{ssec:terminalset}.}
	\label{fig:terminalset}
\end{figure*}

Using the approach in \cite{Gruber2021CSL} implemented in the MATLAB toolbox AROC \cite{AROC}, we obtain the safe terminal set $\zono{\matzeros{},G}$ whose generator matrix $G$ (see \cref{fig:safeterminalset} in the Appendix) is square and full-rank.
Hence, the set $\zono{\matzeros{},G}$ is a parallelotope and can be easily converted into a polytope $\targetset{}\subset \R{12}$, as required by \cref{alg:BRSEA_ti}.
We use the generator matrix
\begingroup\setlength{\arraycolsep}{4pt}
\begin{equation*}
	G = \tfrac{1}{5}
	\begin{bmatrix}
	1 & 0 & 0 & 1 & -1 		     & 1 & -1 			& 0 & \phantom{-}0 \\
	0 & 1 & 0 & 1 & \phantom{-}1 & 0 & \phantom{-}0 & 1 & -1 \\
	0 & 0 & 1 & 0 & \phantom{-}0 & 1 & \phantom{-}1 & 1 & \phantom{-}1 \\
	\end{bmatrix}
\end{equation*}
\endgroup
to define the input set and disturbance set as \cite[Sec.~V-D]{Gruber2023TAC}
\begin{equation*}
	\inputset{} = [-9.81, 2.38] \times \zono{\matzeros{}, \zeta G} \subset \R{4}, \;
	\distset{} = \zono{\matzeros{}, \varphi G} \subset \R{3} ,
\end{equation*}
where the scaling factors $\zeta \in \R{}$ and $\varphi \in \R{}$ allow us to compare the results for different input and disturbance capacities:
$\zeta^{(1)} = 0.5$ and $\varphi^{(1)} = 0$,
$\zeta^{(2)} = 1$ and $\varphi^{(2)} = 0$, and
$\zeta^{(3)} = 1$ and $\varphi^{(3)} = 0.05$.
We set $\tau = [0,1]$.

\cref{fig:terminalset} shows various projections of the time-interval EA backward reachable set $\innerBRSEA{-\tau}$ corresponding to the different values of $\zeta$ and $\varphi$, with the computation times in \cref{tab:allresults}.
We observe that $\innerBRSEAsuper{-\tau}{2} \supseteq \innerBRSEAsuper{-\tau}{1}$, which is due to the enlarged input capacity in the second case, showing that more input capacity can steer additional states into the target set, thereby enlarging the EA backward reachable set.
Similarly, we have $\innerBRSEAsuper{-\tau}{3} \subseteq \innerBRSEAsuper{-\tau}{2}$, as the third case incorporates disturbances.
Since the input capacities are equal in both cases, we observe that enlarging the disturbance shrinks the size of the EA backward reachable set.



\subsection{Scalability Analysis}
\label{ssec:scalability}

{\setlength{\arraycolsep}{2pt}
Finally, we analyze the scalability of our backward reachability algorithms by means of the scalable platoon benchmark \cite{BenMakhlouf2014ARCH2}, whose dynamics are given in \cite[Eq.~(9)]{BenMakhlouf2014ARCH2}, where we choose $\gamma = 2$ as in \cite[Sec.~2.4]{BenMakhlouf2014ARCH2}.
For a number of trucks $\numTrucks{} \in \N{}$, the state vector is $x(t) = \begin{bmatrix} x^{(1)}(t)^\top & \dotsc & x^{(\numTrucks{})}(t)^\top \end{bmatrix}^\top \in \R{3\numTrucks{}}$ with $x^{(j)}(t) = \begin{bmatrix} e^{(j)}(t) & \dot{e}^{(j)}(t) & a^{(j)}(t) \end{bmatrix}^\top$, where $e^{(j)}(t)$ is the relative position between trucks $j-1$ and $j$ shifted by a safe distance, $\dot{e}^{(j)}(t)$ is the relative velocity between trucks $j-1$ and $j$, and $a^{(j)}(t)$ is the acceleration of the $j$th truck.
The input $u(t) \in \R{\numTrucks{}}$ concatenates the input accelerations $u^{(j)}$ of all $\numTrucks{}$ trucks, and the disturbance $w(t) \in \R{}$ is the acceleration of the leading truck.
}

We use $t=2$ and $\tau = [0,2]$ for the time-point and time-interval backward reachable sets, respectively, and $\steps = 100$ steps.
The target set $\targetset{} \subset \R{3\numTrucks{}}$ and the input set $\inputset{} \subset \R{\numTrucks{}}$ are given by the Cartesian product over the sets for each truck.
The individual target sets are $\targetsetupper{j} = \polyHRep{\polymat{},\polyvec{}}$, with
\begingroup
\setlength{\arraycolsep}{2pt}
\setcounter{MaxMatrixCols}{20}
\begin{equation*}
    \polymat{}^\top = 
    \begin{bmatrix}
        1 & -1 & -1 & 0 & 0  & 0 & 0  \\
        0 &  0 & -2 & 1 & -1 & 0 & 0  \\
        0 &  0 & 0  & 0 & 0  & 1 & -1 
    \end{bmatrix}
\end{equation*}
and $\polyvec{}^\top = \begin{bmatrix} 0 & 20 & 7 & 10 & -3 & 5 & -1 \end{bmatrix}$ for AE sets
and $\polyvec{}^\top = \begin{bmatrix} 20 & 0 & 0 & 1.5 & 1.5 & 1 & 1 \end{bmatrix}$ for EA sets.
\endgroup
We bound the input acceleration of each truck by $\inputsetupper{j} = [-5,1]\si{\meter\per\square\second}$.
The acceleration of the leading truck is $\distset{} = [-0.5,0.5]\si{\meter\per\square\second}$.

\cref{tab:platoon} lists the computation times of all four time-point and time-interval backward reachable sets for an increasing number of trucks $\numTrucks{}$.
The computation of $\outerBRSAE{-t}$ is always fastest since it is the only algorithm that scales with $\bigO{n^3}$.
Second is the other time-point solution $\innerBRSEA{-t}$ due to only one operation being $\bigO{n^{4.5}}$.
Compared to the time-point solutions, the computation of both time-interval solutions is more time-consuming, largely due to the numerous linear programs and concatenation of large zonotope generator matrices.
The evaluation of the scalable platoon benchmark demonstrates the polynomial runtime complexity in the state dimension of all our backward reachability algorithms, enabling the analysis of very high-dimensional linear systems.

\begin{table}[t]
	\centering \small
	\caption{Computation times (timeout: $100\si{\second}$) for the platoon benchmark for increasing state dimension $n$ and input dimension $m$.}
	\label{tab:platoon}
	\begin{tabular}{c c c c c c}
		\toprule
		$n$ & $m$ & $\outerBRSAE{-t}$ & $\outerBRSAE{-\tau}$ & $\innerBRSEA{-t}$ & $\innerBRSEA{-\tau}$ \\
		\midrule
		15 & 5     & $0.01\si{\second}$ & $2.2\si{\second}$  & $0.06\si{\second}$ & $0.22\si{\second}$ \\
		51 & 17    & $0.01\si{\second}$ & $7.8\si{\second}$  & $0.09\si{\second}$ & $0.26\si{\second}$ \\
		99 & 33    & $0.03\si{\second}$ & $70\si{\second}$   & $0.43\si{\second}$ & $2.2\si{\second}$ \\
		150 & 50   & $0.07\si{\second}$ & ---   			 & $0.68\si{\second}$ & $5.2\si{\second}$ \\
		300 & 100  & $0.42\si{\second}$ & ---   			 & $2.7\si{\second}$  & $22\si{\second}$ \\
		600 & 200  & $2.3\si{\second}$  & ---   			 & $13\si{\second}$   & --- \\
		999 & 333  & $11\si{\second}$   & ---			   	 & $45\si{\second}$   & --- \\
		2001 & 667 & $84\si{\second}$   & ---			     & ---				  & --- \\
		\bottomrule
	\end{tabular}
\end{table}




\subsection{Discussion}
\label{ssec:discussion}

Let us now address some critical aspects regarding our proposed backward reachability algorithms:
First of all, the target set $\targetset{}$ must be represented as a polytope, see \cref{def:polytope}.
One can easily design polytopes manually;
however, if the target set is the result of another algorithm and it is not represented as a polytope, one is forced to enclose it by a polytope (for minimal reachability) or find a polytope that is contained in the original set (for maximal reachability)---both cases can be handled via optimization.

As discussed in the respective subsections, the approximation errors of all backward reachable sets, except the time-point EA backward reachable set, are non-zero even in the limit $\Delta t \to 0$.
Obtaining rigorous convergence results would require bounds in terms of the Hausdorff distance between the two sides of several set-based inequalities, including \eqref{eq:reordering}, \cref{lmm:convMinkDiff}, and \cref{prop:union}, which represent challenging problems left to future work.
Still, one can tighten the time-point and time-interval AE backward reachable sets in arbitrary directions by additional support function evaluations.
For the time-interval EA backward reachable set, the approximation error entirely depends on the tightness of the containment in \cref{prop:union}.
For large disturbances, the forward reachable set of a given initial state may not be contained within the target set at any specific point in time, but still pass through the target set over a time interval.
In this case, the initial state would be part of the time-interval solution, but not of any time-point solution.
Further investigation into this issue is required to formally capture the notion of one set passing through another, different from both containment and intersection.

Since we know that there exists a control input to steer each state of the EA backward reachable set into the target set, a natural next step is the extraction of such a controller as in \cite[Sec.~IV-B.2]{Yang2022CSL}.
The sets in our work are limited to feed-forward controllers because we consider the effects of the control input and disturbance separately.
Instead, one can also skip backward reachability and directly synthesize a controller, which is a well-researched topic for linear continuous-time systems offering a wide range of different approaches.


\section{Conclusion}
\label{sec:conclusion}

This article presents the first backward reachability algorithms using set propagation techniques for perturbed continuous-time linear systems.
The proposed algorithms cover minimal and maximal reachability and compute both time-point and time-interval solutions.
The runtime complexity of all algorithms is polynomial in the state dimension.
Our evaluation shows tight results and how changes in the input and disturbance set affect the size of the resulting backward rechable set.
Furthermore, we examined the scalability of our algorithms by analyzing systems with well over a hundred state variables within seconds, which significantly improves the state of the art in backward reachability analysis.



\appendix





\noindent \emph{Proof of \cref{prop:BRSAE_tp}}: \\
We have
\begin{align*}
	&x_0 \in e^{-At} \big( ( \targetset{} \oplus -\ZW{t} ) \ominus \ZU{t} \big) \\
	&\Leftrightarrow \forall z_u \in \ZU{t}\colon e^{At} x_0 + z_u \in \targetset{} \oplus -\ZW{t} \\
	&\Leftrightarrow \forall z_u \in \ZU{t} \; \exists z_w \in \ZW{t}\colon e^{At}x_0 + z_u + z_w \in \targetset{} \\
	&\Leftrightarrow \forall u(\cdot) \in \inputsignals{} \; \exists w(\cdot) \in \distsignals{}\colon \trajx{t;x_0,u(\cdot),w(\cdot)} \in \targetset ,
\end{align*}
which is equal to the definition in \cref{eq:def_BRSAE_tp}.
\hfill $\square$

\noindent \emph{Proof of \cref{prop:minkSum_polyzono}}: \\
We insert $\poly{} \oplus \Z{}$ into \eqref{eq:sFset} to obtain
\begin{align*}
	&\poly{} \oplus \Z{} \subseteq \polyHRep{\polymat{},\polyvectilde{}}, \\
	&\forall j \in \Nint{1}{\cons{}}\colon \polyvectilde{(j)} = \sF{\poly{} \oplus \Z{},\polymat{(j,\cdot)}^\top} \\
	&\hspace{71pt} = \sF{\poly{},\polymat{(j,\cdot)}^\top} + \sF{\Z{},\polymat{(j,\cdot)}^\top} \\
	&\hspace{71pt} = \polyvec{(j)} + \sF{\Z{},\polymat{(j,\cdot)}^\top} .
\end{align*}
The runtime complexity follows from the $\cons{}$ support function evaluations of $\Z{}$, see \eqref{eq:sF_zono} and \cref{tab:setops}.
\hfill $\square$


\noindent \emph{Proof of \cref{thm:BRSAE_ti}}: \\
By considering only a finite subset of input trajectories $\inputsignalssub{} \subset \inputsignals{}$, we obtain an outer approximation:
\begin{align}
	\BRSAE{-\tau} \overset{\eqref{eq:BRSAE_ti}}&{=}
	\bigcap_{u^* \in \inputsignals{}} \BRSE{-\tau;u^*(\cdot)} \nonumber \\
	&\subseteq \bigcap_{u^* \in \inputsignalssub{}} \BRSE{-\tau;u^*(\cdot)} \eqqcolon \S{1}.
	\label{eq:BRSAE_ti_inputsignalssub}
\end{align}
%
Let us denote the input trajectory $\forall t \in \tau\colon u(t) = \centerOp{\inputset{}}$ by $u_0$ and the other $\numInp{}$ input trajectories in $\inputsignalssub{}$ by $u_1, ..., u_{\numInp{}}$.
To evaluate $\S{1}$ in \eqref{eq:BRSAE_ti_inputsignalssub}, we compute an outer approximation of $\BRSE{-\tau;u_0}$ that also encloses $\BRSAE{-\tau}$ since
\begin{equation*}
	\BRSAE{-\tau} \overset{\eqref{eq:BRSAE_ti_inputsignalssub}}{\subseteq}
	\BRSE{-\tau;u_0} \overset{\eqref{eq:outerBRSE}}{\subseteq}
	\outerBRSE{-\tau;u_0}.
\end{equation*}
Second, we incorporate all other input trajectories in $\inputsignalssub{}$:
\begin{align}
	\S{1} &= \bigcap_{j \in \{0,...,\numInp{}\}} \BRSE{-\tau;u_j} \nonumber \\
	\overset{\eqref{eq:outerBRSE_k}}&{\subseteq} \big( \outerBRSE{-\tau_0;u_0} \cup ... \cup \outerBRSE{-\tau_{\steps{}-1};u_0} \big) \nonumber \\
	&\qquad \cap \BRSE{-\tau;u_1} \cap ... \cap \BRSE{-\tau;u_{\numInp{}}} \nonumber \\
	\begin{split} \label{eq:union_intersection}
	\overset{\eqref{eq:outerBRSE}}&{\subseteq} \big( \outerBRSE{-\tau_0;u_0} \cup ... \cup \outerBRSE{-\tau_{\steps{}-1};u_0} \big) \\
	&\qquad \cap \outerBRSE{-\tau;u_1} \cap ... \cap \outerBRSE{-\tau;u_{\numInp{}}} \eqqcolon \S{2} .
	\end{split}
\end{align}
We enclose each additional set by the polytope constructed using support function evaluations in the directions $\dir{1}, ..., \dir{\numInp{}}$:
\begin{align}
\begin{split} \label{eq:sF_BRSE}
	&\forall j \in \Nint{1}{\numInp{}}\colon
	\outerBRSE{-\tau;u_j} \overset{\eqref{eq:sFset}}{\subseteq} \polyHRep{N,p^{(j)}} \\
	&\text{with} \; N = [\dir{1} ... \dir{\numInp{}}]^\top, \forall i \in \Nint{1}{\numInp{}}\colon p^{(j)}_{(i)} = \sF{\outerBRSE{-\tau;u_j},\dir{j}} .
\end{split}
\end{align}
We insert this in \eqref{eq:union_intersection} to obtain
\begin{align*}
	\S{2} \overset{\eqref{eq:sF_BRSE}}&{\subseteq}
	\big( \outerBRSE{-\tau_0;u_0} \cup ... \cup \outerBRSE{-\tau_{\steps{}-1};u_0} \big) \\
	&\qquad \cap \polyHRep{N,p^{(1)}} \cap ... \cap \polyHRep{N,p^{(\numInp{})}} \\
	&= \big( \outerBRSE{-\tau_0;u_0} \cup ... \cup \outerBRSE{-\tau_{\steps{}-1};u_0} \big) \cap \polyHRep{N,p} .
\end{align*}
We obtain $p = \min_{j \in \{1,...,\numInp{}\}} p^{(j)}$ element-wise by construction, see \eqref{eq:innerZtraj}.
Finally, distributing the intersection over the union yields $\outerBRSAE{-\tau}$ in \eqref{eq:outerBRSAE_ti}.
\hfill $\square$

\noindent \emph{Proof of \cref{prop:BRSEA_tp}}: \\
We have
\begin{align*}
	&x_0 \in e^{-At} \big( ( \targetset{} \ominus \ZW{t} ) \oplus -\ZU{t} \big) \\
	&\Leftrightarrow \exists z_u \in \ZU{t}\colon e^{At} x_0 + z_u \in \targetset{} \ominus \ZW{t} \\
	&\Leftrightarrow \exists z_u \in \ZU{t} \; \forall z_w \in \ZW{t}\colon e^{At}x_0 + z_u + z_w \in \targetset{} \\
	&\Leftrightarrow \exists u(\cdot) \in \inputsignals{} \; \forall w(\cdot) \in \distsignals{}\colon \trajx{t;x_0,u(\cdot),w(\cdot)} \in \targetset ,
\end{align*}
which is equal to the definition in \cref{eq:def_BRSEA_tp}.
\hfill $\square$


\noindent \emph{Proof of \cref{lmm:convMinkDiff}}: \\
We plug into the definitions of the Minkowski difference \eqref{eq:def_minkDiff} and convex hull \eqref{eq:def_conv}:
\begin{align*}
	&\convOp{\S{1} \ominus \S{3},\S{2} \ominus \S{3}} \oplus \S{3} \\
	&\; = \{ \lambda a + (1-\lambda) b + c \, | \, \lambda \in [0,1], a \oplus \S{3} \subseteq \S{1}, \\
	&\hspace{103pt} b \oplus \S{3} \subseteq \S{2}, c \in \S{3} \} \\
	&\; = \{ \lambda (a + c) + (1-\lambda) (b + c) \, | \, \lambda \in [0,1], a \oplus \S{3} \subseteq \S{1}, \\
	&\hspace{135pt} b \oplus \S{3} \subseteq \S{2}, c \in \S{3} \} \\
	&\; \subseteq \{ \lambda s_1 \oplus (1-\lambda) s_2 \, | \, \lambda \in [0,1], s_1 \in a \oplus \S{3} \subseteq \S{1}, \\
	&\hspace{95pt} s_2 \in b \oplus \S{3} \subseteq \S{2} \} \\
	&\; \subseteq \convOp{\S{1},\S{2}} .
\end{align*}
Using the identity $(\S{} \oplus \S{3}) \ominus \S{3} = \S{}$ \cite[Lemma~1(iii)]{Yang2022CSL} yields the claim.
\hfill $\square$

\noindent \emph{Proof of \cref{thm:BRSEA_ti}}: \\
We can expand the right-hand side of \eqref{eq:BRSEA_tp_union} to
\begin{equation*}
	\big \{ \big( e^{-At} \targetset{} \ominus e^{-At} \ZW{t} \big) \oplus e^{-At} (-\ZU{t}) \, \big| \, t \in \tau_k \big \} \eqqcolon \S{1} ,
\end{equation*}
and insert
\begin{align*}
	\{ e^{-At} \targetset{} \, | \, t \in \tau_k \}
		\overset{\eqref{eq:Hti}}&{=} e^{-At_{k+1}} \Hti{\tau_0} \\
	\{ e^{-At} \ZW{t} \, | \, t \in \tau_k \}
		\overset{\eqref{eq:outerZ},\eqref{eq:Zprop}}&{\subseteq} e^{-At_{k+1}} \outerZW{\tau_k} \\
	\{ e^{-At} (-\ZU{t}) \, | \, t \in \tau_k \}
		\overset{\eqref{eq:innerZ},\eqref{eq:Zprop}}&{\supseteq} e^{-At_{k+1}} (-\innerZU{t_k})
\end{align*}
to obtain
\begin{equation*}
	\S{1} \supseteq e^{-At_{k+1}} \big( ( \Hti{\tau_0} \ominus \ZW{\tau_k} ) \oplus -\ZU{\tau_k} \big) \eqqcolon \S{2} .
\end{equation*}
%
%
%
Next, we replace $\Hti{\tau_0}$ by its inner approximation, see \eqref{eq:innerHti}:
\begin{align*}
	&\S{2} \supseteq e^{-At_{k+1}} \big( \big( ((\convOp{\targetset{},e^{A\Delta t}\targetset{}} \ominus \F{} \, \boxOp{\targetset{}}) \\
	&\hspace{73pt} \ominus \B{\mu}) \ominus \outerZW{\tau_k} \big) \oplus -\innerZU{t_k} \big) \eqqcolon \S{3} .
\end{align*}
Note that we enclose $\targetset{}$ by $\boxOp{\targetset{}}$ to evaluate the multiplication with the interval matrix $\F{}$ using \eqref{eq:intmatzono} and compute $\mu$ as in \eqref{eq:mu} using the generator matrix of $\boxOp{\targetset{}}$.
We now apply \cref{lmm:convMinkDiff} and convert the two polytopes of the convex hull operation to constrained zonotopes by \cref{alg:conZonoOp} to efficiently evaluate the Minkowski sum with $-\innerZU{t_k}$:
\begin{align*}
	&\S{3} \supseteq e^{-A t_{k+1}} \big( -\innerZU{t_k} \, \oplus \\
	&\hspace{25pt} \operator{conv} \big( ((\conZonoOp{ \targetset{} \ominus \F{} \, \boxOp{\targetset{}}) \ominus \B{\mu}) \ominus \outerZW{\tau_k} } , \\
	&\hspace{25pt} \conZonoOp{ ((e^{A \Delta t} \targetset{} \ominus \F{} \, \boxOp{\targetset{}}) \ominus  \B{\mu}) \ominus \outerZW{\tau_k} } \big) \big) \\
	&\hspace{10pt} \eqqcolon \innerBRSEA{-\tau_k} .
\end{align*}
Thus, each set $\innerBRSEA{-\tau_k}$ is an inner approximation of the union of time-point solutions over $\tau_k$, which in turn is an inner approximation of the time-interval solution $\BRSEA{-\tau_k}$:
\begin{equation*}
	\innerBRSEA{-\tau_k} \subseteq \bigcup_{t \in \tau_k} \BRSEA{-t} \overset{\text{\cref{prop:union}}}{\subseteq} \BRSEA{-\tau_k} .
\end{equation*}
Extending this reasoning to all $\steps{}$ consecutive time intervals yields the claim.
\hfill $\square$


\begin{figure*}
	\centering
	\begin{tikzpicture}[scale = 1, every node/.style={font=\scriptsize}]
		\node[anchor=west] at (0,0) {
			\setcounter{MaxMatrixCols}{12}
			\begingroup\setlength\arraycolsep{2pt}
			$G =
			\begin{bmatrix*}[S]
				-0.0042 & 0.0455 & 0.0064 & -0.0694 & 0 & 0 & 0.0001 & -0.0004 & 0 & 0 & -0.0002 & -0.0004 \;\; \\
				0.0455 & 0.0042 & 0.0694 & 0.0064 & 0 & 0 & 0.0004 & 0.0001 & 0 & 0 & -0.0004 & 0.0002 \\
				0 & 0 & 0 & 0 & -0.0370 & 0.0377 & 0 & 0 & 0 & 0 & 0 & 0 \\
				0.0086 & -0.0924 & 0.0031 & -0.0331 & 0 & 0 & 0.0008 & -0.0022 & 0 & 0 & -0.0003 & -0.0006 \\
				-0.0924 & -0.0086 & 0.0331 & 0.0031 & 0 & 0 & 0.0022 & 0.0008 & 0 & 0 & -0.0006 & 0.0003 \\
				0 & 0 & 0 & 0 & 0.0491 & 0.0284 & 0 & 0 & 0 & 0 & 0 & 0 \\
				-0.0044 & -0.0004 & 0.0083 & 0.0008 & 0 & 0 & 0.0088 & 0.0032 & 0 & 0 & 0.0046 & -0.0023 \\
				0.0004 & -0.0044 & 0.0008 & -0.0083 & 0 & 0 & 0.0032 & -0.0088 & 0 & 0 & 0.0023 & 0.0046 \\
				0 & 0 & 0 & 0 & 0 & 0 & 0 & 0 & 0.0045 & -0.0005 & 0 & 0 \\
				-0.0091 & -0.0008 & 0.0071 & 0.0007 & 0 & 0 & -0.0244 & -0.0088 & 0 & 0 & 0.0016 & -0.0008 \\
				0.0008 & -0.0091 & 0.0007 & -0.0071 & 0 & 0 & -0.0088 & 0.0244 & 0 & 0 & 0.0008 & 0.0016 \\
				0 & 0 & 0 & 0 & 0 & 0 & 0 & 0 & -0.0019 & -0.0011 & 0 & 0
			\end{bmatrix*}$
			\endgroup
		};
	\end{tikzpicture}
	\caption{Generator matrix $G$ of the safe terminal set $\zono{\matzeros{},G}$ for the quadrotor system in \cref{ssec:terminalset} computed using the approach in \cite{Gruber2021CSL}.}
	\label{fig:safeterminalset}
\end{figure*}

\section*{Acknowledgment}
Many thanks to our colleagues Adrian Kulmburg, Tobias Ladner, Lukas Schäfer, and Victor Ga{\ss}mann for their help in the formalization of some proofs, the design and evaluation of the numerical examples, and the discussion of the algorithms.


\AtNextBibliography{\small}
\printbibliography


\begin{IEEEbiography}[{\includegraphics[width=1in,height=1.25in,clip,keepaspectratio]{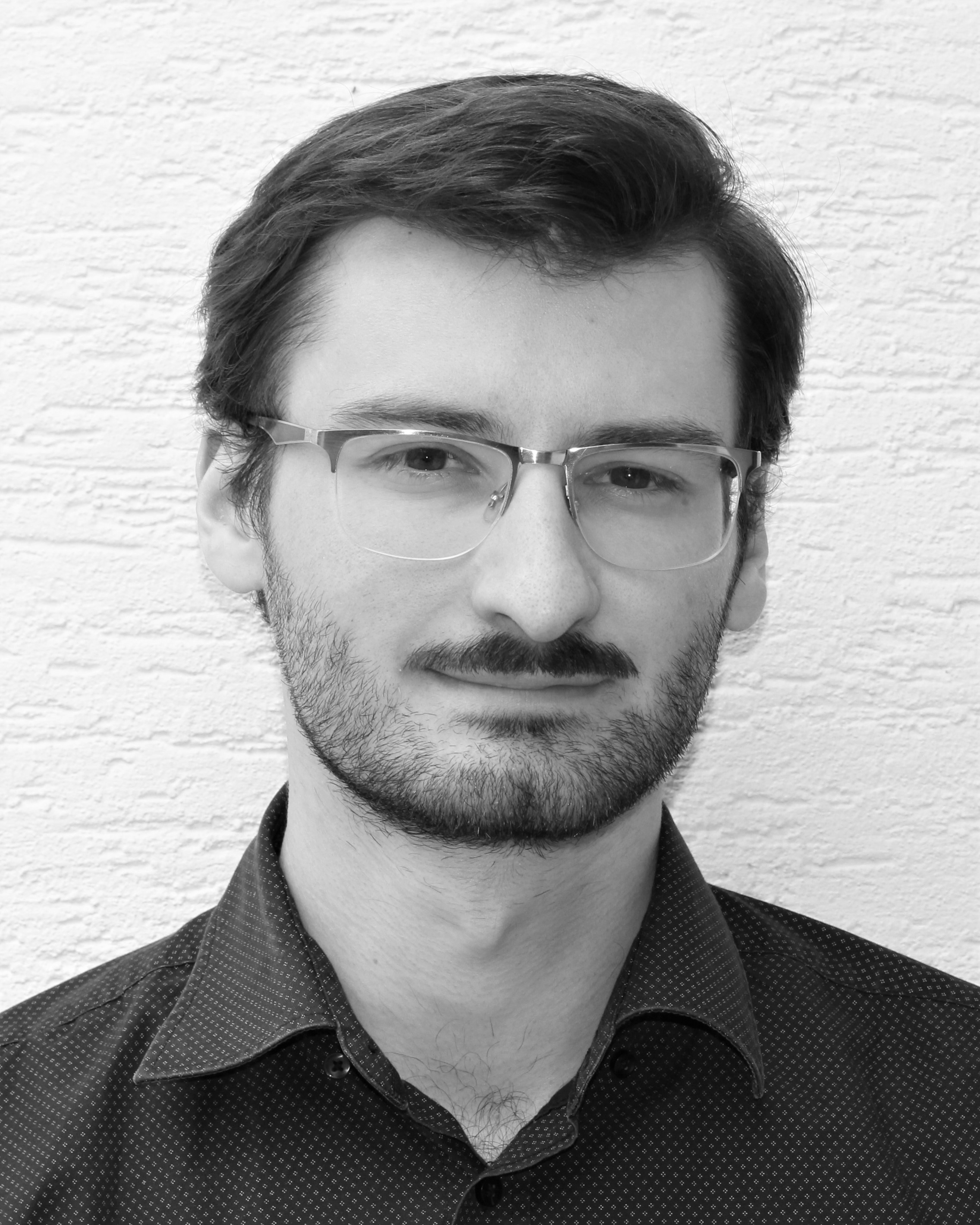}}]{MARK WETZLINGER}{\space} received the B.S. degree in Engineering Sciences in 2017 jointly from Universit\"{a}t Salzburg, Austria and Technische Universit\"{a}t M\"{u}nchen, Germany, and the M.S. degree in Robotics, Cognition and Intelligence in 2019 from Technische Universit\"{a}t M\"{u}nchen, Germany, and his Ph.D. degree in computer science in 2024 at Technische Universit\"{a}t M\"{u}nchen, Germany. His research interests include formal verification of linear and nonlinear continuous systems, reachability analysis, adaptive parameter tuning, and model order reduction.
\end{IEEEbiography}

\begin{IEEEbiography}[{\includegraphics[width=1in,height=1.25in,clip,keepaspectratio]{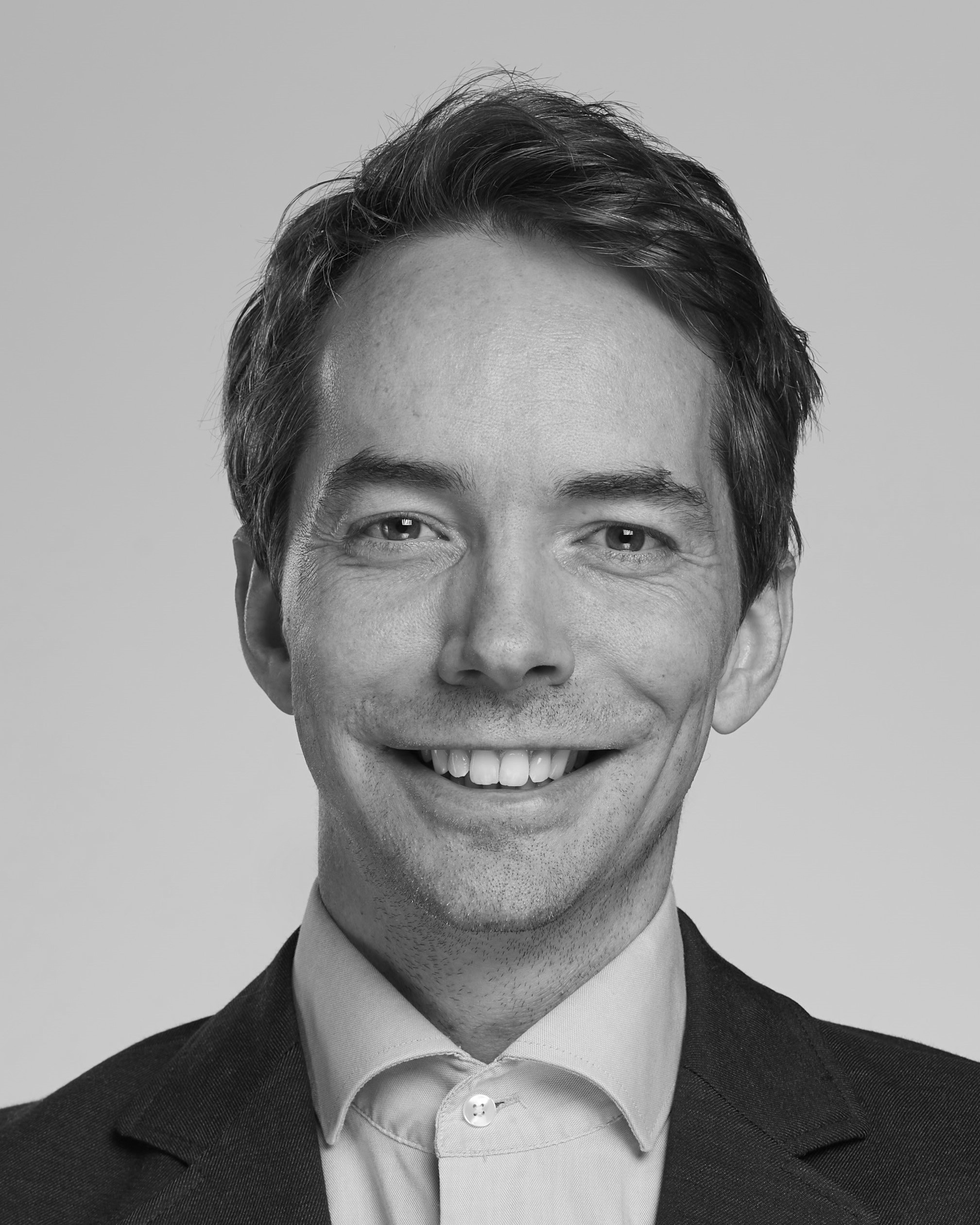}}]{MATTHIAS ALTHOFF}{\space} is an associate professor in computer science at Technische Universit\"{a}t M\"{u}nchen, Germany. He received his diploma engineering degree in Mechanical Engineering in 2005, and his Ph.D. degree in Electrical Engineering in 2010, both from Technische Universit\"{a}t M\"{u}nchen, Germany. From 2010 to 2012 he was a postdoctoral researcher at Carnegie Mellon University, Pittsburgh, USA, and from 2012 to 2013 an assistant professor at Technische Universit\"{a}t Ilmenau, Germany. His research interests include formal verification of continuous and hybrid systems, reachability analysis, planning algorithms, nonlinear control, automated vehicles, and power systems.
\end{IEEEbiography}

\end{document}